\documentclass[12pt]{article}
\usepackage{amssymb}
\usepackage{amsmath}
\usepackage{amsthm}
\usepackage{comment}
\usepackage{graphicx}
\begin{document}

\def\cO{\mathcal{O}}
\def\ee{\varepsilon}

\newcommand{\removableFootnote}[1]{}

\newtheorem{theorem}{Theorem}[section]
\newtheorem{corollary}[theorem]{Corollary}
\newtheorem{proposition}[theorem]{Proposition}

\theoremstyle{definition}
\newtheorem{definition}{Definition}[section]



\title{
Fast phase randomisation via two-folds
}
\author{
D.J.W.~Simpson$^{\dagger}$ and M.R.~Jeffrey$^{\ddagger}$\\\\
$^{\dagger}$Institute of Fundamental Sciences\\
Massey University\\
Palmerston North\\
New Zealand\\\\
$^{\ddagger}$Department of Engineering Mathematics\\
University of Bristol\\
Bristol\\
UK
}
\maketitle

\begin{abstract}

A two-fold is a singular point on the discontinuity surface of a piecewise-smooth vector field,
at which the vector field is tangent to the discontinuity surface on both sides.
If an orbit passes through an invisible two-fold
(also known as a Teixeira singularity)
before settling to regular periodic motion,
then the phase of that motion cannot be determined from initial conditions,
and in the presence of small noise
the asymptotic phase of a large number of sample solutions is highly random.
In this paper we show how the probability distribution of the asymptotic phase
depends on the global nonlinear dynamics. 
We also show how the phase of a smooth oscillator can be randomised
by applying a simple discontinuous control law that generates an invisible two-fold.
We propose that such a control law can be used to desynchronise a collection of oscillators,
and that this manner of phase randomisation is fast compared to existing
methods (which use fixed points as phase singularities) because there is
no slowing of the dynamics near a two-fold.


\end{abstract}

\section{Introduction}
\label{sec:intro}
\setcounter{equation}{0}


When a solution of a system of ordinary differential equations approaches a stable periodic orbit, rather than its location in phase space at a given time, a more useful quantity is often the asymptotic (or eventual) phase of the orbit. The asymptotic phase distinguishes the long-term behaviour of different orbits, and sets of points with the same asymptotic phase are called {\it isochrons}. 
If different isochrons intersect they create a {\it phaseless point}, where the asymptotic phase is undefined.
Near a phaseless point the asymptotic phase is highly sensitive to perturbations.
In smooth systems, phaseless points are often unstable equilibria
which can only be approached asymptotically (in backward time).
In piecewise-smooth systems there exist phaseless points 
that can be intersected in finite time, without any slowing of the dynamics suffered near an equilibrium.
The two-fold singularity, an enigma of piecewise-smooth dynamical systems theory, is one such phaseless point. 

There are many reasons why it might be desirable to alter the phase
of an existing oscillatory motion. 
One approach to tackling Parkinson's disease, for instance,
is to disturb the synchronised neural activity associated with physical tremors \cite{HaBe07,NiFe95}.
Desynchronisation can be achieved with pulses \cite{Ta02}, pulse trains \cite{Ta03},
a control that utilises time-delay \cite{PoHa06}, or some other well-chosen feedback law \cite{FrCh12}.
Phase randomisation for prototypical neuron models is achieved in \cite{DaHe09} 
by briefly applying a control that transports the orbit to the close proximity of a phaseless point.
The orbit subsequently returns to periodic motion but now has a different asymptotic phase.
This phase is highly sensitive to the precise point at which the orbit is located when the control is removed, 
so small stochasticity in the system causes the resulting asymptotic phases of different neurons to be highly randomised.

In piecewise smooth dynamical systems, two-fold singularities were first described by Filippov \cite{Fi88},
and have garnered interest as points where both the forward and backward time uniqueness of a flow can break down, in an otherwise deterministic flow \cite{Te90,Je11,CoJe13}.
For a vector field that is discontinuous on some hypersurface -- the {\it discontinuity surface} --
a two-fold is a point where the flow has quadratic tangencies (`folds') to both sides of the surface simultaneously.
Two-folds occur generically at isolated points in three-dimensional piecewise-smooth vector fields, 
and on $(n-2)$-dimensional manifolds in $n$-dimensional vector fields \cite{CoJe13}.
They have been identified in models of circuit systems \cite{DiCo11} and mechanical systems \cite{SzJe14},
and have deep connections to folded nodes of slow-fast systems \cite{DeJe11}. 
The dynamics local to a two-fold depends on whether the vector field on either side is curving toward
or away from the discontinuity surface,  and on the alignments of the two vector fields relative to each other \cite{CoJe13}. 

Here we illustrate the practical implications of a two-fold for phase randomisation.
As well as suggesting how two-folds might affect real systems,
this suggests a use for them as control devices for the desynchronisation of oscillators. To this end we shall consider systems where a two-fold occurs 
either naturally in a piecewise-smooth system, or is introduced to a smooth system via a discontinuous control.
We will simulate orbits in the presence of small noise as they pass close to a two-fold, 
and focus on how their phases are randomised by the nearby presence of the singularity. 
The two-fold is not a zero of the vector field, so there is no slowing of the flow during phase randomisation achieved in this way.

The central result of the paper is summarised by Fig.~\ref{fig:phaseHistograms} (to be produced in \S\ref{sec:stochDyns}), showing the phase distributions of $10^4$ sample solutions from a fixed initial point,  simulated passing through a two-fold singularity. 
The central result of the paper is summarised by Fig.~\ref{fig:phaseHistograms} (to be produced in \S\ref{sec:stochDyns}), showing the phase distributions of $10^4$ sample solutions from a fixed initial point,  simulated passing through a two-fold singularity. 
To produce these figures we shall simulate the normal form of the two-fold,  
add higher order terms such that orbits emanating from the two-fold approach a stable periodic orbit, and add small noise.
Fig.~\ref{fig:phaseHistograms}-A shows a histogram of the phase $\phi_T$, computed relative to a reference time $T$ and limiting to the asymptotic phase as $T \to \infty$.
This shows an apparently uniform distribution, hence a fully randomised phase of solutions from the same initial point. 
Fig.~\ref{fig:phaseHistograms}-B shows an analogous histogram using different higher order terms, showing how nonlinearity of the flow can be used to influence the distribution. 
The black curves show the theoretical approximations we derive for the probability density functions of the phase.


\begin{figure}[t!]
\begin{center}
\setlength{\unitlength}{1cm}
\begin{picture}(16.5,4)
\put(0,0){\includegraphics[height=4cm]{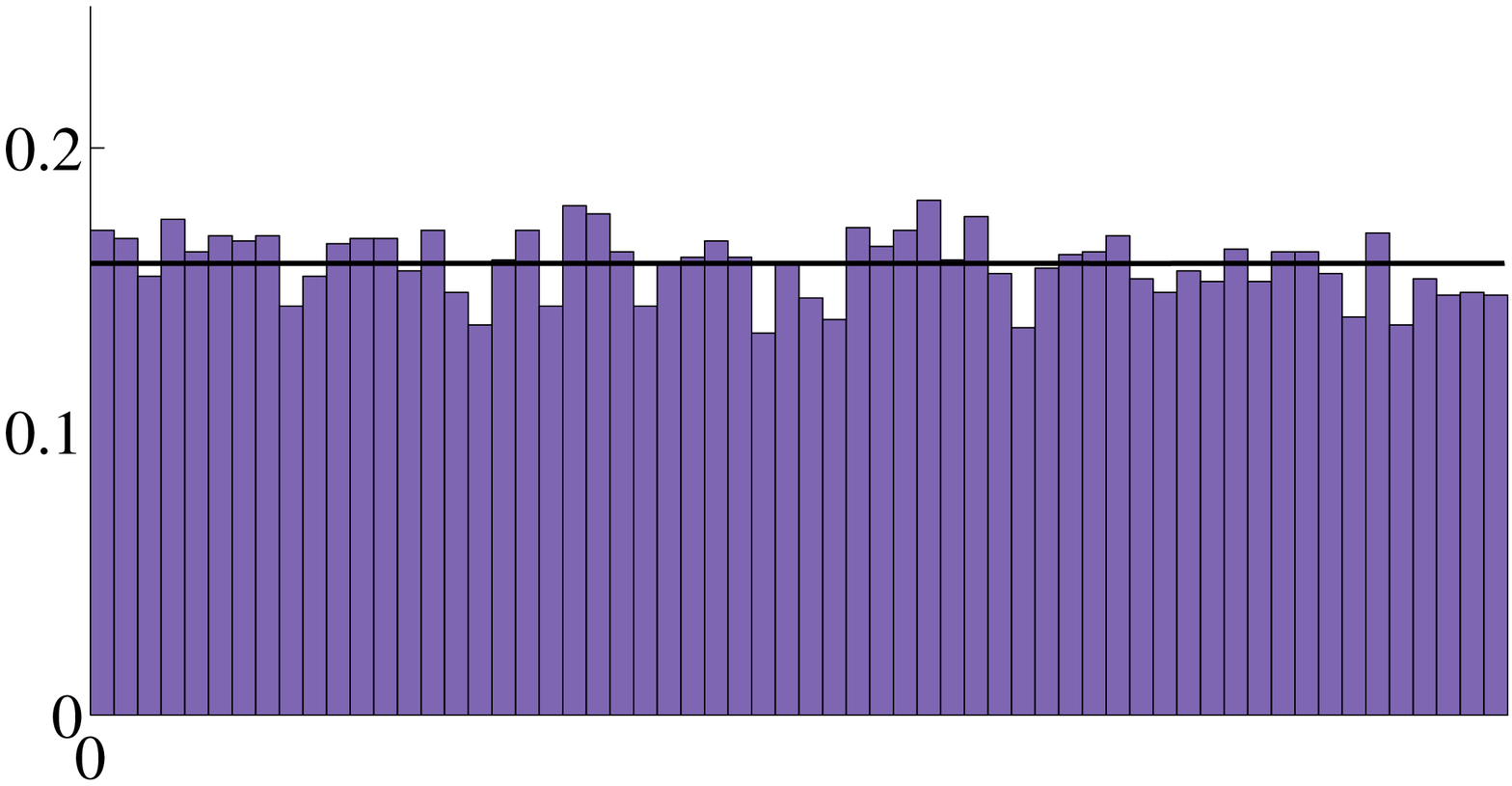}}
\put(8.5,0){\includegraphics[height=4cm]{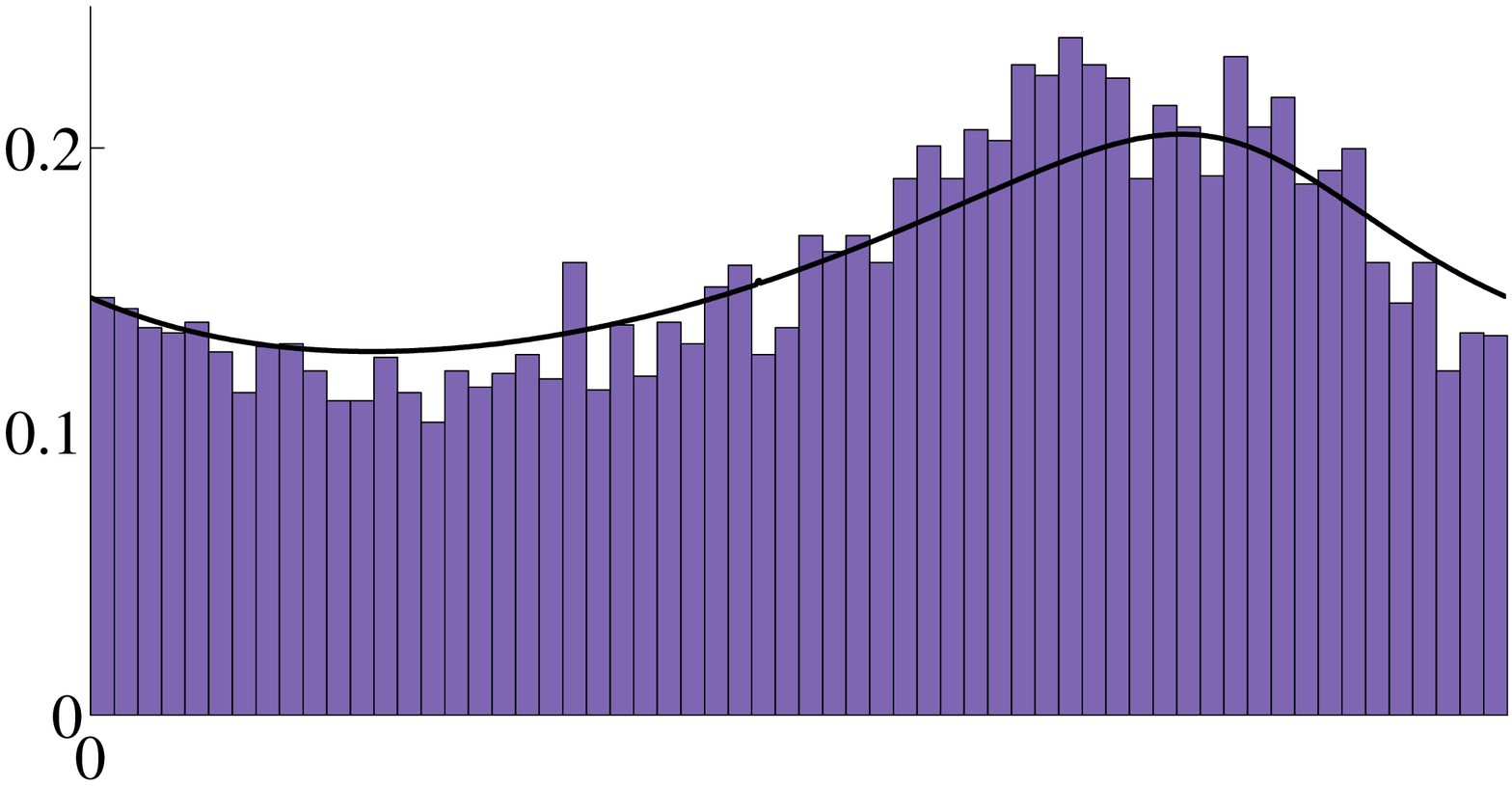}}
\put(1,3.9){\large \sf A}
\put(9.5,3.9){\large \sf B}
\put(5,-.1){\small $\phi_T$}
\put(0,2.5){\small $p_T$}
\put(8.5,2.5){\small $p_T$}
\put(2.47,.06){\footnotesize $\frac{\pi}{2}$}
\put(4.22,.1){\footnotesize $\pi$}
\put(5.97,.06){\footnotesize $\frac{3 \pi}{2}$}
\put(7.72,.1){\footnotesize $2 \pi$}
\put(13.5,-.1){\small $\phi_T$}
\put(10.97,.06){\footnotesize $\frac{\pi}{2}$}
\put(12.72,.1){\footnotesize $\pi$}
\put(14.47,.06){\footnotesize $\frac{3 \pi}{2}$}
\put(16.22,.1){\footnotesize $2 \pi$}
\end{picture}
\caption{
Each panel shows a histogram of the phase $\phi_T$,
of $10^4$ sample solutions from the same initial condition, 
of the general stochastic system (\ref{eq:sde}).
The two panels correspond to different choices for the higher order terms
which affect the global dynamics of (\ref{eq:sde}).
One sample solution corresponding to each panel is shown in Fig.~\ref{fig:phaseqq23} (see caption of this figure for the parameter values used).
The solid curves show a theoretical approximation for the
probability density function of $\phi_T$ as derived in \S\ref{sub:pdf}.
\label{fig:phaseHistograms}
} 
\end{center}
\end{figure}

Building on this, we propose that a periodic orbit in a smooth system can be made to take an excursion through a two-fold, by applying a control action, before returning to smooth periodic motion with a randomised phase. We provide simulations of a simple model in which the effect is to desynchronise a collection synchronised oscillators. 
As discussed in \S\ref{sec:conc},
two-folds have potential advantages over equilibria as phaseless points in the context of phase randomisation.
Although we provide a motivating toy example, calculating and optimising realisable control actions (as in e.g. \cite{DaHe09}) are beyond the scope of this paper. 



In \S\ref{sec:normalForm} we briefly set out the normal form equations of the two-fold singularity,
and review a few pertinent details of the local dynamics, including a novel definition of polar coordinates that naturally fit with the local dynamics (which should be of particular interest for dynamicists studying two-folds). 
We include higher order terms with the normal form in \S\ref{sec:perturbed}, and extend these definitions
before simulating a flow through the two-fold subject to stochastic perturbations in \S\ref{sec:stochDyns}. 
In \S\ref{sec:desynch} we give an example of an application to desynchronising smooth oscillators
using a discontinuous control, with some closing remarks in \S\ref{sec:conc}.

\section{Deterministic dynamics of the normal form}
\label{sec:normalForm}
\setcounter{equation}{0}


There exist three main kinds of two-fold singularity, depending on whether both folds are visible (curving away the discontinuity surface),
or invisible (curving toward the discontinuity surface), or one is visible and one is invisible. In each case 
there exist local conditions under which the flow traverses the singularity in finite time. 
In this paper we focus on the case where both folds are invisible, called an invisible two-fold or Teixeira singularity, in which 
the {\it entire} local flow traverses the singularity in finite time. The added interest of the invisible two-fold is that the local flow winds repeatedly around the singularity, giving oscillatory behaviour. 


%


In three dimensions where $X = (x,y,z)$,
the normal form of the invisible two-fold is the piecewise-smooth system
\begin{equation}
\dot{X} =
\begin{cases} 
F_L(X) \;, & x < 0\;, \\
F_R(X) \;, & x > 0\;,
\end{cases} 
\label{eq:ode}
\end{equation}
where
\begin{equation}
F_L(X) = (z,V^-,1) \;, \qquad
F_R(X) = (-y,1,V^+) \;,
\label{eq:FLFR}
\end{equation}
and $V^-, V^+ \in \mathbb{R}$ are parameters, \cite{JeCo09,CoJe11}.
The two-fold is located at the origin, $X = (0,0,0)$.
The discontinuity surface $x=0$
consists of attracting and repelling sliding regions denoted $A$ and $R$ (where the flow is confined to the surface and so slides along it),
and two crossing regions $C^\pm$ (where the flow passes transversally though the surface),
illustrated in Fig.~\ref{fig:twoFoldSchematic}. 

\begin{figure}[b!]
\begin{center}
\setlength{\unitlength}{1cm}
\begin{picture}(12,6)
\put(0,0){\includegraphics[height=6cm]{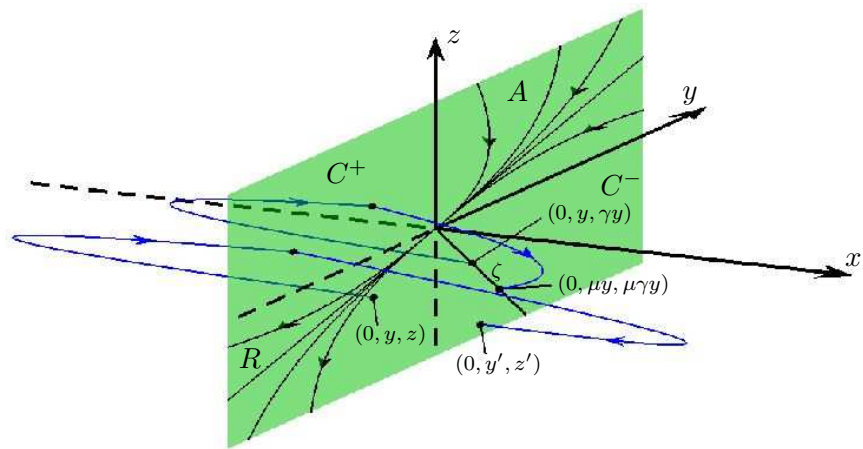}}
\put(11.4,2.5){\small $x$}
\put(9.24,4.72){\small $y$}
\put(6.1,5.46){\small $z$}
\put(4.88,1.52){\scriptsize $(0,y,z)$}
\put(6.22,1.12){\scriptsize $(0,y',z')$}
\put(7.44,3.12){\scriptsize $(0,y,\gamma y)$}
\put(7.58,2.18){\scriptsize $(0,\mu y,\mu \gamma y)$}
\put(6.7,2.36){\scriptsize $\zeta$}
\put(6.9,4.7){\small $A$}
\put(3.34,1.14){\small $R$}
\put(4.5,3.6){\small $C^+$}
\put(8.15,3.45){\small $C^-$}
\end{picture}
\caption{
A schematic of the two-fold of (\ref{eq:ode}).
The discontinuity surface $x=0$ is made up of an attracting sliding region $A$ ($y,z > 0$),
a repelling sliding region $R$ ($y,z < 0$), and two crossing sliding regions,
$C^+$ ($y < 0, z > 0$) and $C^-$ ($y > 0, z < 0$).
Parts of two orbits and their intersections with the discontinuity surface are shown.
One orbit starts in $R$ and has subsequent intersections in $C^+$ and $C^-$.
A second orbit starts and returns to an invariant line $\zeta \in C^-$, see \S\ref{sub:returnMap}.
\label{fig:twoFoldSchematic}
}
\end{center}
\end{figure}

The system is solved across the discontinuity by forming the Filippov system for (\ref{eq:ode}).
Using $e_i$ to denote the $i^{\rm th}$ coordinate vector in $\mathbb{R}^3$,
a {\it Filippov solution} of (\ref{eq:ode}) 
is an absolutely continuous function $\varphi(t)$ that satisfies
\begin{equation}
\begin{split}
\frac{d \varphi}{d t} &= F_L(\varphi(t)) \;, \quad {\rm if~} e_1^{\sf T} \varphi(t) < 0 \;, \\
\frac{d \varphi}{d t} &= F_R(\varphi(t)) \;, \quad {\rm if~} e_1^{\sf T} \varphi(t) > 0 \;, \\
\frac{d \varphi}{d t} &\in \left\{ (1-q) F_L(\varphi(t)) + q F_R(\varphi(t))
~\middle|~ 0 \le q \le 1 \right\} \;,
\quad {\rm if~} e_1^{\sf T} \varphi(t) = 0 \;,
\end{split}
\label{eq:differentialInclusion}
\end{equation}
for almost all
values of $t$ \cite{Fi88}.
We let $\varphi(t;X_0,t_0)$ denote a Filippov solution to (\ref{eq:ode})
with initial condition $\varphi(t_0) = X_0$.
Each $\varphi(t;X_0,t_0)$ is a concatenation of several smooth orbit segments,
including segments of evolution outside the discontinuity surface $x=0$, 
and segments of evolution on $x=0$ referred to as {\it sliding} motion. 
Solving the convex combination in (\ref{eq:differentialInclusion}) for $\dot x=0$ gives the system that sliding motion obeys on $A$ and $R$ as
\begin{equation}
\left[ \begin{array}{c} \dot{y} \\ \dot{z} \end{array} \right] =
\frac{1}{y+z} \left[ \begin{array}{cc}
V^- & 1 \\
1 & V^+
\end{array} \right]
\left[ \begin{array}{c} y \\ z \end{array} \right] \;.
\label{eq:slidingVectorField}
\end{equation}
Various details of the sliding flow can be derived from (\ref{eq:slidingVectorField}),
see \cite{JeCo09,CoJe11}.
Throughout this paper we assume
\begin{equation}
V^- < 0 \;, \qquad
V^+ < 0 \;, \qquad
V^- V^+ > 1 \;,
\label{eq:VmVp}
\end{equation}
in which case a typical orbit of (\ref{eq:differentialInclusion}) and (\ref{eq:slidingVectorField})
will intersect the two-fold in forward time or backward time or both (see e.g. \cite{JeCo09}),
giving the geometry in Fig.~\ref{fig:twoFoldSchematic}.
With (\ref{eq:VmVp}), the matrix in (\ref{eq:slidingVectorField}) has distinct negative eigenvalues,
the smaller of which is $\lambda = \frac{1}{2} ( V^+ + V^- + [{( V^+ - V^- )^2 + 4}]^{1/2} )$.
As orbits of (\ref{eq:ode}) slide into the singularity inside $A$,
or slide out of the singularity in $R$,
they do so tangent to the weak eigenvector of the matrix 
(``weak'' meaning associated with the smallest eigenvalue)
and so are tangent to the line $z = (\lambda - V^-) y$ on $x = 0$.
The time to reach or depart the singularity is finite.
Specifically, the time is $t = (V^- -\lambda- 1)y/\lambda$
from an initial point on the line $z = (\lambda - V^-) y$ with $x = 0$.

In the remainder of this section we
review crossing solutions to (\ref{eq:ode}) in \S\ref{sub:returnMap},
introduce polar coordinates for describing the flow as it spirals away from the two-fold in \S\ref{sub:polar},
and define the notion of ``viable'' Filippov solutions in the context of simulation in \S\ref{sub:viable}.
In later sections we apply these concepts to more general piecewise-smooth systems.

\subsection{Crossing dynamics and an unstable cone}
\label{sub:returnMap}

The left and right half systems of (\ref{eq:ode}),
$\dot{X} = F_L(X)$ and $\dot{X} = F_R(X)$,
have solutions
\begin{equation}
\varphi_L(t;X_0,t_0) = \left(
x_0 + z_0 (t-t_0) + \frac{1}{2} (t-t_0)^2 ,\,
y_0 + V^- (t-t_0) ,\,
z_0 + t - t_0 \right) \;,
\label{eq:leftSolution}
\end{equation}
and
\begin{equation}
\varphi_R(t;X_0,t_0) = \left(
x_0 - y_0 (t-t_0) - \frac{1}{2} (t-t_0)^2 ,\,
y_0 + t - t_0 ,\,
z_0 + V^+ (t-t_0) \right) \;,
\label{eq:rightSolution}
\end{equation}
respectively, where $X_0 = (x_0,y_0,z_0)$.
In what follows we use (\ref{eq:leftSolution}) and (\ref{eq:rightSolution})
to study Filippov solutions $\varphi(t;0,y,z,t_0)$ of (\ref{eq:ode}).
Orbits wind around the regions $A$ and $R$, making repeating crossings of $x=0$ via the regions $C^\pm$, which can be understood with a return map
that has been studied in considerable detail \cite{CoJe11,FeGa12}.
Here we review a few details needed for the ensuing analysis. 

First consider the region $z < 0$ on $x=0$, where the flow $\varphi_L$ is directed away from the discontinuity surface. 
If $y > 0$, then $(0,y,z) \in C^-$, see Fig.~\ref{fig:twoFoldSchematic}, $\varphi_R$ is directed toward the surface so $\varphi(t;0,y,z,t_0)$ immediately enters $x < 0$.
If instead $y < 0$, then $(0,y,z) \in R$ where both $\varphi_L$ and $\varphi_R$ are directed away from the surface, so $\varphi(t;0,y,z,t_0)$ is not uniquely determined
(the may slide along $x=0$ following (\ref{eq:slidingVectorField}), or enter either $x<0$ or $x>0$ at any instant). 
For definiteness, when $y \le 0$ we ignore solutions that follow (\ref{eq:slidingVectorField}) through $R$, and consider the Filippov solution
$\varphi(t;0,y,z,t_0)$ that immediately enters $x < 0$.
From (\ref{eq:leftSolution}) we find that $\varphi(t;0,y,z,t_0)$ resides in $x < 0$
until the later time $t = t_0 - 2 z$ when it is located at $(0,y - 2 V^- z,-z)$.

Second, consider $y < 0$ and assume $\varphi(t;0,y,z,t_0)$ immediately enters $x > 0$.
From (\ref{eq:rightSolution}) and applying analogous arguments to the case $z<0$, we find that $\varphi(t;0,y,z,t_0)$ resides in $x > 0$
until $t = t_0 - 2 y$ when it is located at $(0,-y,-2 V^+ y + z)$.

From these observations we can construct a return map capturing crossing dynamics, formulated in the following proposition (see \cite{CoJe11} for a complete proof).

\begin{proposition}
Suppose $V^-, V^+ < 0$, $V^- V^+ > 1$.
For any $y \in \mathbb{R}$ and $z < 0$,
let $(0,y',z')$ denote the next intersection of $\varphi(t;0,y,z,t_0)$
with the discontinuity surface at a point with $y' > 0$.
If $z < \frac{2 V^+}{4 V^- V^+ - 1} \,y$, then $z' < 0$ and
\begin{equation}
\begin{bmatrix} y' \\ z' \end{bmatrix} =
\begin{bmatrix} -1 & 2 V^- \\ -2 V^+ & 4 V^- V^+ - 1 \end{bmatrix}
\begin{bmatrix} y \\ z \end{bmatrix} \;,
\label{eq:returnMap}
\end{equation}
the smooth orbit from $(y,z)$ to $(y',z')$ taking a time  
\begin{equation}
t - t_0 = 2 \left( (2 V^- - 1) z - y \right) \;,
\label{eq:returnTime}
\end{equation}
whereas if $z > \frac{2 V^+}{4 V^- V^+ - 1} \,y$, then $z' > 0$.
\label{pr:returnMap}
\end{proposition}

The return map (\ref{eq:returnMap}) is linear and has a saddle-type fixed point at $(y,z) = (0,0)$.
The unstable multiplier
(i.e. the eigenvalue of the matrix in (\ref{eq:returnMap}) with modulus greater than $1$) is
\begin{equation}
\mu = 2 V^- V^+ \left( 1 + \sqrt{1 - \frac{1}{V^- V^+}} \right) - 1 \;,
\label{eq:mu}
\end{equation}
and the corresponding eigenvector is $(1,\gamma)$ where
\begin{equation}
\gamma = V^+ \left( 1 + \sqrt{1 - \frac{1}{V^- V^+}} \right) \;.
\label{eq:gamma}
\end{equation}
In the full three dimensional space occupied by (\ref{eq:ode}), this eigenvector corresponds to a ray
\begin{equation}
\zeta = \left\{ (0,y,\gamma y) ~\big|~ y>0 \right\} \;,
\label{eq:zeta}
\end{equation}
that emanates from the two-fold.

By evolving the flow forward from $\zeta$ according to (\ref{eq:leftSolution}) and (\ref{eq:rightSolution})
we obtain a surface $\Lambda$ that is given implicitly by
\begin{equation}
x = \begin{cases}
\frac{-1}{V^-} \Xi(y,z) \;, & x \le 0 \\
\frac{1}{V^+} \Xi(y,z) \;, & x \ge 0
\end{cases} \;,
\label{eq:Lambda}
\end{equation}
where
\begin{equation}
\Xi(y,z) = \frac{V^+ y^2 - 2 V^+ V^- y z + V^- z^2}{2 (V^+ V^- - 1)} \;.
\label{eq:Xi}
\end{equation}
Fig.~\ref{fig:unstableManifold} shows a plot of $\Lambda$.
$\Lambda$ is non-differentiable at $x=0$ and consists of
one conical portion on each side of $x=0$.

%

\begin{figure}[t!]
\begin{center}
\setlength{\unitlength}{1cm}
\begin{picture}(8,6)
\put(0,0){\includegraphics[height=6cm]{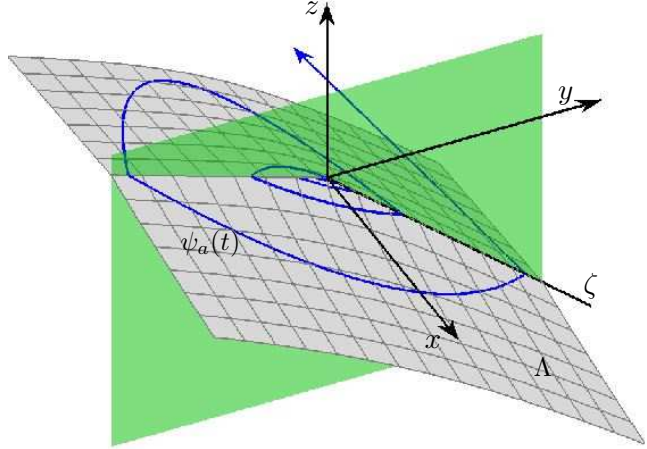}}
\put(5.65,1.37){\small $x$}
\put(7.42,4.69){\small $y$}
\put(4.05,5.81){\small $z$}
\put(7.76,2.1){\small $\zeta$}
\put(2.4,2.67){\footnotesize $\psi_a(t)$}
\put(7.1,1){\footnotesize $\Lambda$}
\end{picture}
\caption{
A plot of (\ref{eq:Lambda}) using $V^- = -0.5$ and $V^+ = -2.5$.
This is the unstable manifold $\Lambda$ (\ref{eq:Lambda}).
One Filippov solution $\psi_a(t)$ is shown in blue.
\label{fig:unstableManifold}
}
\end{center}
\end{figure}

For any initial point on the ray $\zeta$, 
the next intersection of $\varphi(t;0,y,\gamma y,t_0)$
with $C^-$ is $(0,\mu y,\mu \gamma y)$, shown in Fig.~\ref{fig:twoFoldSchematic}.
By substituting $z = \gamma y$ into (\ref{eq:returnTime}), 
we find that this iteration of (\ref{eq:returnMap}) corresponds to an evolution time of
\begin{equation}
t - t_0 = 2 \left( (2 V^- - 1) \gamma - 1 \right) y \;.
\label{eq:returnTimeEig}
\end{equation}
By performing this iteration backward through repeated crossings,
we can construct a Filippov solution that 
approaches the two-fold in backward time.
By taking the limit of infinitely many intersections,
we obtain a Filippov solution that emanates from the two-fold and
has infinitely many intersections with $\zeta$ (in finite time).
Given any $a > 0$, let $\psi_a(t)$ (defined for $t \ge 0$) denote the
unique Filippov solution that is located at the two-fold at $t = 0$,
and located on $\zeta$ at $t = a$.
That is, $\psi_a(0) = (0,0,0)$ and $\psi_a(a) \in \zeta$.
Such an orbit is shown in Fig.~\ref{fig:unstableManifold}.

By using (\ref{eq:returnTimeEig}) and the classical formula for the sum of a geometric series,
we determine the $y$-value of $\psi_a(a)$ to be
\begin{equation}
e_2^{\sf T} \psi_a(a) = \frac{1}{1+\frac{1-\frac{1}{V^-}}{\sqrt{1-\frac{1}{V^- V^+}}}} \,a \;.
\label{eq:yValueOnZeta}
\end{equation}

After $t = a$, $\psi_a(t)$ next intersects $\zeta$ at $t = \mu a$.
Consequently $\psi_a \left( \mu^n a \right) \in \zeta$, for all $n \in \mathbb{Z}$.
It follows that each $\psi_{\mu^n a}(t)$ is the same orbit as $\psi_a(t)$,
and so we can characterise all the $\psi_a(t)$ by restricting our attention to values of $a$ in an interval
$a^* \le a < \mu a^*$, for some $a^* > 0$.

We can therefore write
\begin{equation}
\Lambda = \left\{ \psi_a(t) ~\middle|~ t \ge 0, 1 \le a < \mu \right\} \;.
\label{eq:Lambda2}
\end{equation}
The intersection of $\Lambda$ with $C^-$ is $\zeta$.
Since $\zeta$ corresponds to an unstable eigenvector of the crossing map,
we refer to $\Lambda$ as an unstable manifold of the two-fold.

\subsection{Polar coordinates}
\label{sub:polar}

On $\Lambda$, orbits of (\ref{eq:ode}) spiral out from the two-fold, as in Fig.~\ref{fig:unstableManifold}.
For this reason it is natural to introduce polar coordinates, but the obvious spherical or cylindrical coordinates bear no relation to the system's behaviour and yield complicated expression that give no insight into the dynamical system.
Instead let us seek radial and angular coordinates $r(x,y)$ and $\theta(x,y)$ that 
may be interpreted as ``cylindrical polar coordinates'' for orbits winding around the sliding regions, derived from the geometry of $\Lambda$ such that the dynamical system takes a particularly simple form. 

To obtain a suitable notion of the phase $\theta$, we begin by finding the times $\tau_L > 0$ and $\tau_R < 0$
that an orbit takes to travel
from $\zeta$ to a given point in $x < 0$ and in $x > 0$, respectively.

First, choose any $x < 0$ and $y \in \mathbb{R}$.
We let $y_0 > 0$ be such that
by evolving (\ref{eq:ode}) forward from $(0,y_0,\gamma y_0)$ we arrive at $(x,y,z) \in \Lambda$,
for some $z \in \mathbb{R}$, without again intersecting $x = 0$,
and let $\tau_L > 0$ be the corresponding evolution time.
Note that the value of $z$ is given from (\ref{eq:Lambda}) by $x = \frac{-1}{V^-} \Xi(y,z)$.

We have $\varphi_L(\tau_L;0,y_0,\gamma y_0,0) = (x,y,z)$,
and so by (\ref{eq:leftSolution}),
\begin{equation}
x = \gamma y_0 \tau_L + \frac{1}{2} \tau_L^2 \;, \qquad
y = y_0 + V^- \tau_L \;, \qquad
z = \gamma y_0 + \tau_L \;.
\end{equation}
By solving the first two of these equations simultaneously for $y_0$ and $\tau_L$,
we obtain
\begin{equation}
\tau_L = \frac{y - \sqrt{y^2 - \frac{2 V^-}{V^+} x}}{V^- \left( 1 + \sqrt{1 - \frac{1}{V^- V^+}} \right)} \;,
\label{eq:tauL}
\end{equation}
and $y_0 = y - V^- \tau_L$.

Second, choose any $x > 0$ and $y \in \mathbb{R}$.
We let $y_0 > 0$ be such that
by evolving (\ref{eq:ode}) backward from $(0,y_0,\gamma y_0)$ we arrive at $(x,y,z) \in \Lambda$,
for some $z \in \mathbb{R}$, without first reintersecting $x = 0$,
and let $\tau_R < 0$ be the corresponding evolution time.
Here $z$ is given by $x = \frac{1}{V^+} \Xi(y,z)$,
and in the same manner as above we obtain
\begin{equation}
\tau_R = y - \sqrt{y^2 + 2 x} \;,
\label{eq:tauR}
\end{equation}
and $y_0 = y - \tau_R$.

Using these we can now introduce a set of polar coordinates. We let the positive $y$-axis
correspond to $\theta = 0$ and $r = y$, i.e.
\begin{equation}
r(0,y) = y \;, \qquad
\theta(0,y) = 0 \;, \qquad 
{\rm for~all~} y > 0 \;.
\label{eq:theta0}
\end{equation}
The positive $y$-axis corresponds to points on $\zeta$.
Naturally we want to define $\theta$ such that a change of $2\pi$ corresponds to one complete revolution from $\zeta$ back to itself.
By (\ref{eq:returnTimeEig}), forward evolution from $(0,y,\gamma y)$ returns to $\zeta$ in a time
proportional to $y$.
This suggests that we want $\dot{\theta}$ to be inversely proportional to $y$
(or the distance from the two-fold, $r$).
Moreover, the orbit returns to $\zeta$ at the point $(0,\mu y,\mu \gamma y)$.
That is, the value of $r$ increases from $y$ to $\mu y$
and so changes by an amount proportional to $y$,
suggesting $\dot{r}$ should be constant.
In summary, we would like to define $r$ and $\theta$ so that they satisfy (\ref{eq:theta0}) and
\begin{equation}
\dot{r} = \alpha \;, \qquad
\dot{\theta} = \frac{\beta}{r} \;,
\label{eq:odePolar}
\end{equation}
for some constants $\alpha, \beta > 0$.
The following choice of $r$ and $\theta$ achieves this.


\begin{proposition}
Suppose $V^-, V^+ < 0$ and $V^- V^+ > 1$.
Let
\begin{align}
r(x,y) &= \begin{cases}
\alpha y - (\alpha-1) \sqrt{y^2 - \frac{2 V^-}{V^+} x} \;, & x < 0 \\
\alpha y - (\alpha-1) \sqrt{y^2 + 2 x} \;, & x > 0
\end{cases} \;,
\label{eq:r} \\
\theta(x,y) &= \begin{cases}
\frac{\beta}{\alpha} \ln \left( 1 + \frac{\alpha}{\frac{y}{\tau_L(x,y)} - V^-} \right) \;, & x < 0 \\
\frac{\beta}{\alpha} \ln \left( 1 + \frac{\alpha}{\frac{y}{\tau_R(x,y)} - 1} \right) + 2 \pi \;, & x > 0
\end{cases} \;,
\label{eq:theta}
\end{align}
where $\tau_L$ and $\tau_R$ are given by (\ref{eq:tauL}) and (\ref{eq:tauR}), and
\begin{equation}
\alpha = \frac{1}{1+\frac{1-\frac{1}{V^-}}{\sqrt{1-\frac{1}{V^- V^+}}}} \;, \qquad
\beta = \frac{2 \pi \alpha}{\ln(\mu)} \;,
\label{eq:alphabeta}
\end{equation}
where $\mu$ is given by (\ref{eq:mu}).
Then (\ref{eq:r}) and (\ref{eq:theta}) define a continuous bijection from $(x,y) \in \mathbb{R}^2$
to $(r,\theta) \in (0,0) \cup (0,\infty) \times [0,2 \pi)$ that satisfies (\ref{eq:theta0}).
Moreover, under this coordinate change
the restriction of (\ref{eq:ode}) to $\Lambda$ is given by (\ref{eq:odePolar}) with (\ref{eq:alphabeta}).
\label{pr:polar}
\end{proposition}

A constructive proof of Proposition \ref{pr:polar}
using (\ref{eq:leftSolution}) and (\ref{eq:rightSolution})
is given in Appendix \ref{app:polarDerivation}.
Fig.~\ref{fig:contours} illustrates contours of $r(x,y)$ and $\theta(x,y)$ as defined in (\ref{eq:r}) and (\ref{eq:theta}).

Given (\ref{eq:theta}), the negative $y$-axis corresponds to 
$\theta = \frac{2 \pi \ln \left( (1-2 \alpha) \mu \right)}{\ln(\mu)}$.
Also by (\ref{eq:alphabeta}) we have $0 < \alpha < \frac{1}{2}$,
with $\alpha \to 0$ as $V^- V^+ \to 1$
and $\alpha \to \frac{1}{2}$ as $V^- V^+ \to \infty$
(assuming $\frac{V^-}{V^+}$ is fixed as the limit is taken).
Graphically we have found also that $0 < \beta < \pi$.

\begin{figure}[b!]
\begin{center}
\setlength{\unitlength}{1cm}
\begin{picture}(16.5,6)
\put(0,0){\includegraphics[height=6cm]{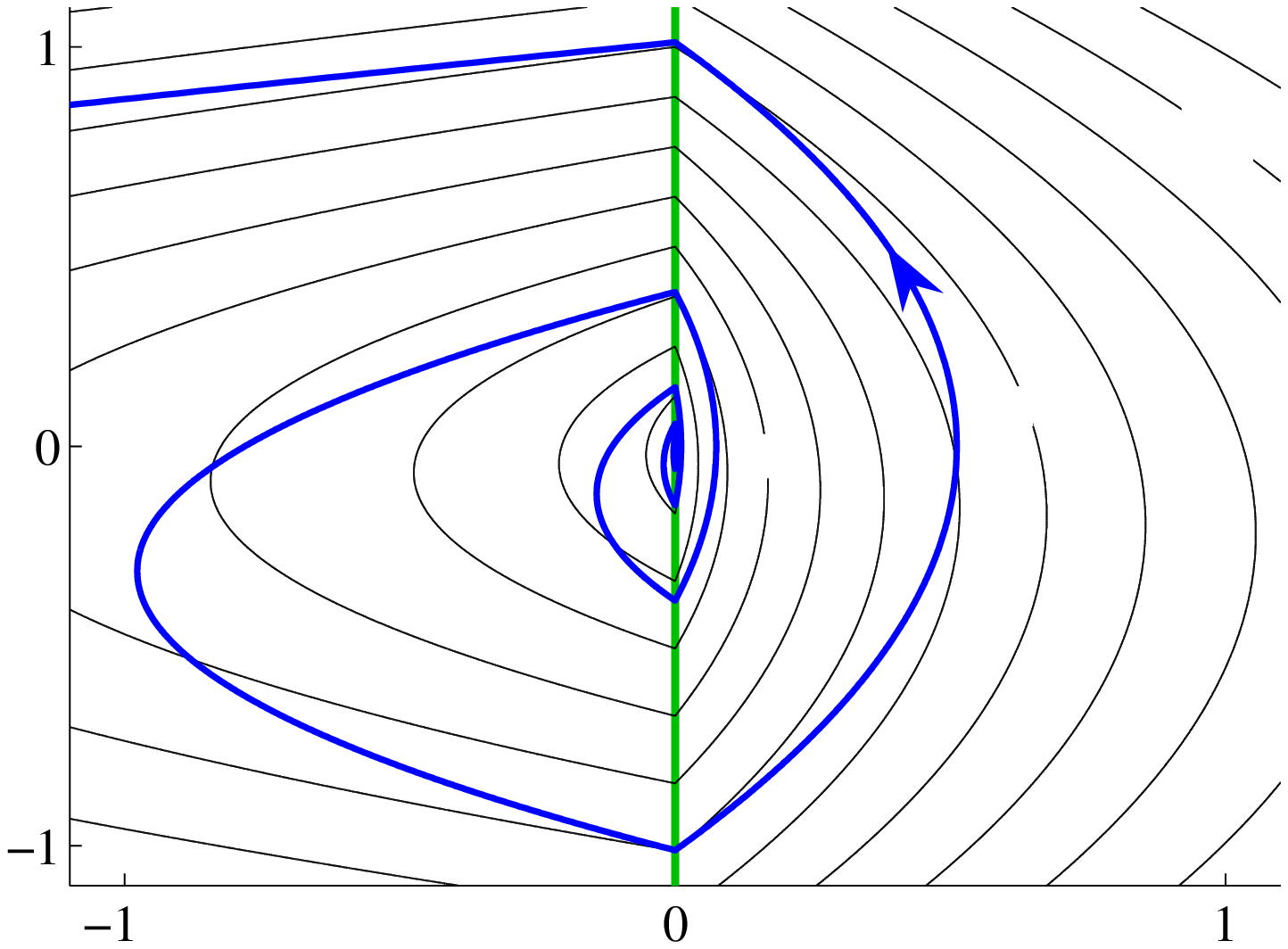}}
\put(8.5,0){\includegraphics[height=6cm]{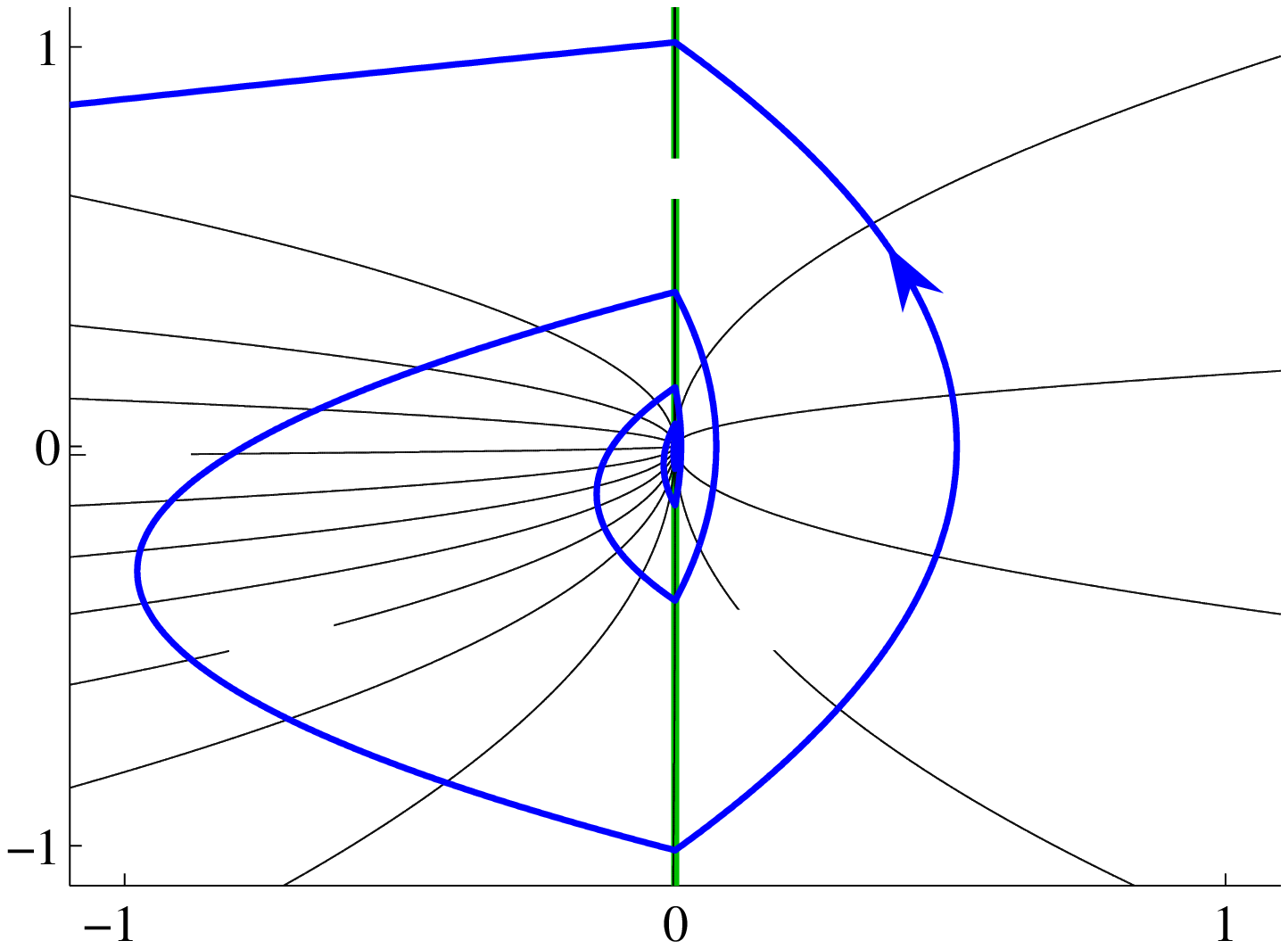}}
\put(4.6,3.2){\tiny $r \hspace{-.5mm}=\hspace{-.7mm} \frac{1}{2}$}
\put(6.15,3.53){\tiny $r \hspace{-.5mm}=\hspace{-.5mm} 1$}
\put(7.25,5.2){\tiny $r \hspace{-.5mm}=\hspace{-.7mm} \frac{3}{2}$}
\put(12.51,4.89){\tiny $\theta \hspace{-.5mm}=\hspace{-.5mm} 0$}
\put(9.2,3.2){\tiny $\theta \hspace{-.5mm}=\hspace{-.7mm} \frac{\pi}{2}$}
\put(10.1,2.08){\tiny $\theta \hspace{-.5mm}=\hspace{-.5mm} \pi$}
\put(13.02,2.12){\tiny $\theta \hspace{-.5mm}=\hspace{-.7mm} \frac{3 \pi}{2}$}
\put(4.15,0){\small $x$}
\put(0,3.24){\small $y$}
\put(12.65,0){\small $x$}
\put(8.5,3.24){\small $y$}
\put(1,6){\large \sf A}
\put(9.5,6){\large \sf B}
\end{picture}
\caption{
Contours $r(x,y)= {\rm constant}$ (\ref{eq:r}) (panel A), and $\theta(x,y)={\rm constant}$ (\ref{eq:theta}) (panel B). 
Contours of $\theta(x,y)$ are parabolas of the form $x = K y^2$, for different values of $K$.
Contours of $r(x,y)$ are parabolas of the form $x = c_1 (y - K)^2 + c_2 K^2$, for different values of $K$,
where $c_1$ and $c_2$ are constants that depend on the values of $V^-$ and $V^+$. 
In each plot one orbit of (\ref{eq:ode}) is shown spiralling away from the origin (shaded blue - colour online). These illustrations use the values $V^- = -0.5$ and $V^+ = -2.5$.
\label{fig:contours}
}
\end{center}
\end{figure}

The right-hand-side of (\ref{eq:odePolar}) is defined everywhere except at $r=0$ (the two-fold)
and can be solved explicitly.
Taking $r \to 0$ as $t \to t_0$ as an initial condition,
for all $t > t_0$ the solution to (\ref{eq:odePolar}) is given by
\begin{equation}
r(t) = \alpha (t-t_0) \;, \qquad
\theta(t) = \left( \frac{\beta}{\alpha} \ln(t-t_0) + C \right) {\rm ~mod~} 2 \pi \;,
\label{eq:solnPolar}
\end{equation}
for some $C \in [0,2 \pi)$.
Therefore the Filippov solution $\psi_a(t)$, defined in \S\ref{sub:returnMap},
may be written in polar coordinates as
\begin{equation}
\psi_a(t) = \left( r(t), \theta(t) \right) =
\left( \alpha t ,\, \frac{\beta}{\alpha} \ln \left( \frac{t}{a} \right) {\rm ~mod~} 2 \pi \right) \;.
\label{eq:solnPolarAlt}
\end{equation}


\subsection{Viable Filippov solutions}
\label{sub:viable}

The following result shows that orbits emanating from the two-fold
either remain in the repelling sliding region $R$ for some time, or they leave $R$ and evolve on $\Lambda$ for all later times.

\begin{proposition}
Let $\varphi(t) = \varphi(t;0,0,0,0)$ be a Filippov solution of (\ref{eq:ode})
located at the two-fold at $t = 0$.
Then either: (i) $\varphi(t) \in \Lambda$ for all $t > 0$,
or (ii) $\varphi(t) \in R$ for some $t > 0$ and $\varphi(t) \not\in \Lambda$ for all $t > 0$.
\label{pr:exit}
\end{proposition}

\begin{proof}
First, suppose $\varphi(t) \in R$ for some $t > 0$.
Since Filippov solutions to (\ref{eq:ode}) can only enter $R$ from the two-fold,
either $\varphi(t) \in R$ for all $t > 0$,
in which case clearly $\varphi(t) \not\in \Lambda$ for all $t > 0$,
or $\varphi(t) \in R$ on time interval $(0,b]$, 
and at $t = b$, $\varphi(t)$ is ejected into either $x < 0$ or $x > 0$.
Suppose without loss of generality that $\varphi(t)$ enters $x < 0$.
Then subsequent crossing motion (for all $t \ge 0$, it can be shown)
is described by the return map (\ref{eq:returnMap}) starting from $\varphi(b)$.
Since $\varphi(b) \not\in \zeta$,
iterations under (\ref{eq:returnMap}) do not lie on $\zeta$,
hence $\varphi(t) \not\in \Lambda$ for all $t > 0$.

Second, suppose $\varphi(t) \not\in R$ for all $t > 0$.
By solving backward in time, the only way that solutions can reach the two-fold without
intersecting $R$ is through intersections with $\zeta$.
This is because orbits cannot reach the two-fold by evolving backward in time on the attracting sliding region.
Nor can solutions reach the two-fold by evolving purely in either the left half-space or the right half-space
because such evolution follows (\ref{eq:leftSolution}) and (\ref{eq:rightSolution}).
The only remaining possibility is to reach the two-fold via crossing motion,
and in view of the saddle-type nature of $(y,z) = (0,0)$ for (\ref{eq:returnMap}),
this can only be achieved along the unstable eigenvector.
Therefore $\varphi(t) \in \Lambda$ for all $t > 0$.
\end{proof}


Now we introduce the notion of a ``viable'' Filippov solution
to represent an orbit that is robust for the purpose of forward time numerical simulation.
Any orbit that evolves from a point in $R$, simulated using discretisation or by modelling the switch using hysteresis, time-delay, or noise,
is likely to be immediately ejected into either $x < 0$ or $x > 0$.
In view of Proposition \ref{pr:exit}, any simulated orbit that arrives at the two-fold
is likely to subsequently evolve on $\Lambda$ (or at least very near $\Lambda$), as observed in \cite{CoJe11}.
Indeed, perturbations of a planar two-fold by hysteresis, time-delay and noise studied in \cite{Si14c}
confirm that perturbed systems follow Filippov solutions provided they do not lie on a repelling sliding region.

\begin{definition}
A Filippov solution is said to {\it viable}
if it does not intersect (or travel along) a repelling sliding region in forward time.
\end{definition}

For the systems studied in this paper, $\varphi(t;X_0,t_0)$ is viable if and only if, for all $t > t_0$,
$\varphi(t;X_0,t_0) \not\in R$.
The idea is that viable Filippov solutions are the most relevant in real systems or in simulations.
The last result of this section follows immediately from Proposition \ref{pr:exit}
and provides us with a complete characterisation of viable Filippov solutions
to (\ref{eq:ode}) that pass through the two-fold.

%

\begin{corollary}
Let $\varphi(t) = \varphi(t;0,0,0,0)$ be a Filippov solution of (\ref{eq:ode})
that is located at the two-fold at $t = 0$. 
Then $\varphi(t)$ is viable if and only if $\varphi(t) = \psi_a(t)$, for some $a$.
\end{corollary}

\section{Deterministic global dynamics}
\label{sec:perturbed}
\setcounter{equation}{0}

The vector field near a generic invisible two-fold in a three-dimensional system
can be transformed to
\begin{equation}
\dot{X} = \begin{cases}
F_L(X) + G_L(X) \;, & x < 0 \\
F_R(X) + G_R(X) \;, & x > 0
\end{cases} 
\label{eq:odePerturbed}
\end{equation}
where $F_L$ and $F_R$ are the constituents of the normal form (\ref{eq:FLFR}), and
\begin{equation}
G_{L,R}(X) = \left( \cO \left( x,|y,z|^2 \right), \cO(|X|), \cO(|X|) \right) \;, 
\label{eq:GLGR}
\end{equation}
represent higher order terms.

As described in \cite{CoJe11}, there exists a surface $\tilde{\Lambda}$
(that coincides with $\Lambda$ in the limit $|X| \to 0$)
on which viable Filippov solutions to (\ref{eq:odePerturbed}) evolve.
This surface intersects $x = 0$ with $y > 0$ on a curve $\tilde{\zeta}$
that matches $\zeta$ to first order.
We let $\tilde{\psi}_a(t)$ denote a viable Filippov solution to (\ref{eq:odePerturbed})
with $\tilde{\psi}_a(0) = 0$ and $\tilde{\psi}_a(a) \in \tilde{\zeta}$.

Next, in \S\ref{sub:phase}, we define the asymptotic phase for orbits of (\ref{eq:odePerturbed}), 
study the associated isochrons in \S\ref{sub:isochrons},
and find times of return to the discontinuity surface in \S\ref{sub:returnTimes}.

\subsection{The asymptotic phase of an orbit}
\label{sub:phase}

Let us suppose that (\ref{eq:odePerturbed}) has a stable periodic orbit $\Gamma$ of period $\tau$
such that orbits emanating from the two-fold on $\tilde{\Lambda}$ approach $\Gamma$.
Examples are given below and also in \cite{CoJe11}.
In the examples considered below,
$\Gamma$ has exactly two intersections with the discontinuity surface,
one with $y > 0$ and one with $y < 0$.
We use the intersection point with $y > 0$ as a reference point with zero phase.

Let $\tilde{\varphi}(t;X_0,t_0)$ be a Filippov solution to (\ref{eq:odePerturbed})
that limits to $\Gamma$ as $t \to \infty$.
The {\it phase} of $\tilde{\varphi}$, relative to a reference time $T$
assumed to be sufficiently large that the orbit lies close to $\Gamma$, will be defined as
\begin{equation}
\phi_T = \frac{2 \pi (T-s_T)}{\tau} {\rm ~mod~} 2 \pi \;, \qquad
s_T = \max_{t \le T} \left[ x(t) = 0 ,\, y(t) > 0 \right] \;.
\label{eq:phiT}
\end{equation}
Note that $s_T \le T$ is the previous time at which the orbit lies on $x=0$ with $y > 0$.
The ``${\rm mod~} 2 \pi$'' ensures $\phi_T = [0,2 \pi)$,
but is almost redundant because the orbit is close to $\Gamma$
and so $T - s_T$ is unlikely to be greater than $\tau$.
Assuming the forward orbit of $X_0$ is unique and converges to $\Gamma$,
the {\it asymptotic phase} of $X_0$ is defined as 
\begin{equation}
\phi = \lim_{T \to \infty} \phi_T \;.
\label{eq:phi}
\end{equation}

\subsection{Isochrons}
\label{sub:isochrons}

An {\it isochron} is a set of points with the same asymptotic phase \cite{Gu75,Wi01}.
Fig.~\ref{fig:isochron} shows five isochrons on $\tilde{\Lambda}$, produced for (\ref{eq:odePerturbed}) with	
\begin{equation}
V^- = -0.5 \;, \qquad V^+ = -2.5 \;,
\label{eq:VmVpEx23}
\end{equation}
and
\begin{equation}
G_L(X) = G_R(X) = -X \;.
\label{eq:GLGRex2}
\end{equation}
$\tilde{\Lambda}$ was computed by fitting a mesh to $1000$ numerically computed forward orbits.
The stable periodic orbit $\Gamma$ forms the boundary of this surface.
The isochrons were computed by interpolating between computed points on the forward orbits
and correspond to $\phi = \frac{2 \pi k}{5}$ for $k = 0,\ldots,4$.
More sophisticated methods for computing isochrons are discussed in \cite{OsMo10,LaKr14}.
Each isochron emanates transversally from $\Gamma$
(as is the case for a generic stable periodic orbit of a smooth system).

\begin{figure}[t!]
\begin{center}
\setlength{\unitlength}{1cm}
\begin{picture}(10,7.5)
\put(0,0){\includegraphics[height=7.5cm]{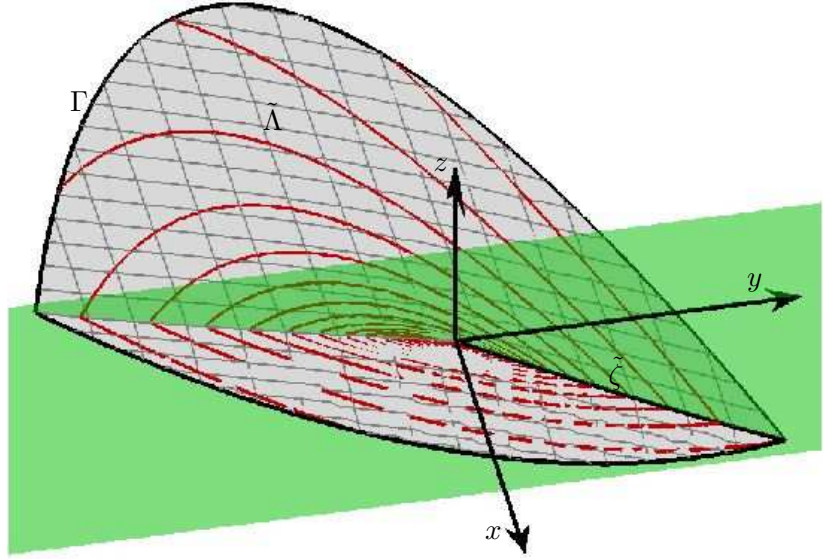}}
\put(6.35,.3){\small $x$}
\put(9.83,3.67){\small $y$}
\put(5.66,5.2){\small $z$}
\put(.83,6){\small $\Gamma$}
\put(3.39,5.75){\small $\tilde{\Lambda}$}
\put(7.99,2.37){\small $\tilde{\zeta}$}
\end{picture}
\caption{
Five isochrons on the surface $\tilde{\Lambda}$ for the system (\ref{eq:odePerturbed})
with (\ref{eq:VmVpEx23}) and (\ref{eq:GLGRex2}).
\label{fig:isochron}
}
\end{center}
\end{figure}

\subsection{Times of return to the discontinuity surface}
\label{sub:returnTimes}

Here we define a return time function $f$ for the curve $\tilde{\zeta}$.
First note that $\tilde{\psi}_a(a) \in \tilde{\zeta}$ by definition.
Then for any $a > 0$, let $t = f(a) > a$ be the next time 
at which $\tilde{\psi}_a(t) \in \tilde{\zeta}$.
Fig.~\ref{fig:phasef2} shows a plot of $f$ using (\ref{eq:VmVpEx23}) and (\ref{eq:GLGRex2})
and was computed by numerical simulation.

For small $a$, $f(a) \approx \mu a$,
because near the two-fold the higher order terms $G_L$ and $G_R$ have little effect
and the return time is $\mu a$ for the normal form (\ref{eq:ode}), as stated in \S\ref{sub:returnMap}.
For large $a$, $\tilde{\psi}_a(t)$ is located near $\Gamma$ and so $f(a) \approx a + \tau$.

In the next section we use $f$ to extrapolate the probability density function
for $\phi_T$ from points near the two-fold to points near $\Gamma$.

\begin{figure}[t!]
\begin{center}
\setlength{\unitlength}{1cm}
\begin{picture}(8,6)
\put(0,0){\includegraphics[height=6cm]{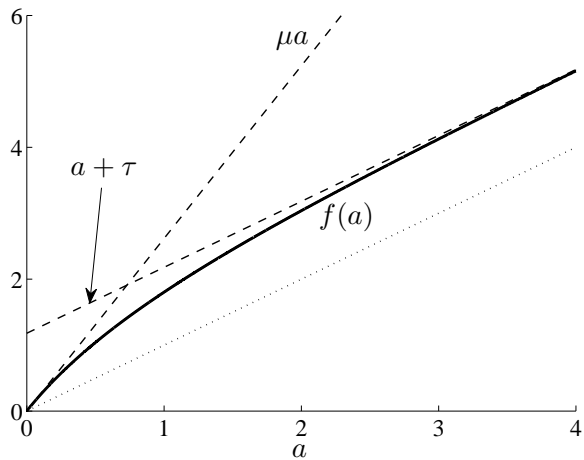}}
\put(4.08,0){\small $a$}
\put(4.44,3.1){\small $f(a)$}
\put(3.87,5.5){\small $\mu a$}
\put(1.14,3.75){\small $a + \tau$}
\end{picture}
\caption{
The return time function $f$ for the system (\ref{eq:odePerturbed})
with (\ref{eq:VmVpEx23}) and (\ref{eq:GLGRex2}).
\label{fig:phasef2}
}
\end{center}
\end{figure}

\section{Stochastic dynamics and phase randomisation}
\label{sec:stochDyns}
\setcounter{equation}{0}

In this section we study a stochastic perturbation of (\ref{eq:odePerturbed}) given by the three-dimensional stochastic differential equation
\begin{equation}
dX(t) =
\left\lbrace
\begin{array}{lc}
F_L(X(t)) + G_L(X(t)) \;, & x(t) < 0 \\
F_R(X(t)) + G_R(X(t)) \;, & x(t) > 0
\end{array}
\right\rbrace dt +
\ee D \,dW(t) \;,
\label{eq:sde}
\end{equation}
where $W(t)$ is a standard three-dimensional Brownian motion,
$D$ is a $3 \times 3$ matrix of constants, and $0 \le \ee \ll 1$ represents the noise amplitude.
Given a sample solution to (\ref{eq:sde}), $\tilde{\varphi}_{\ee}(t;X_0,0)$,
and a time $T$, we define $s_T$ and $\phi_T$ by (\ref{eq:phi}) in the same manner as for the deterministic system (\ref{eq:odePerturbed}).

Here we show that the asymptotic phase is highly randomised for sample solutions to (\ref{eq:sde})
that pass close to the two-fold before approaching a stable periodic orbit.
We provide numerical evidence for this in \S\ref{sub:randomisation},
then derive theoretical approximations to the phase distribution in \S\ref{sub:pdf}.
To illustrate the results we use (\ref{eq:VmVpEx23}) and two choices for $G_L$ and $G_R$,
namely (\ref{eq:GLGRex2}) and
\begin{equation}
G_L(X) = G_R(X) = (-x^3,-y^3,0) \;,
\label{eq:GLGRex3}
\end{equation}
because they provide substantially different phase distributions.
With different values of $V^-$ and $V^+$ and different choices for $G_L$ and $G_R$
we have observed similar results.

\subsection{Sample solutions and phase randomisation}
\label{sub:randomisation}

\begin{figure}[t!]
\begin{center}
\setlength{\unitlength}{1cm}
\begin{picture}(16.5,10)
\put(0,4){\includegraphics[height=6cm]{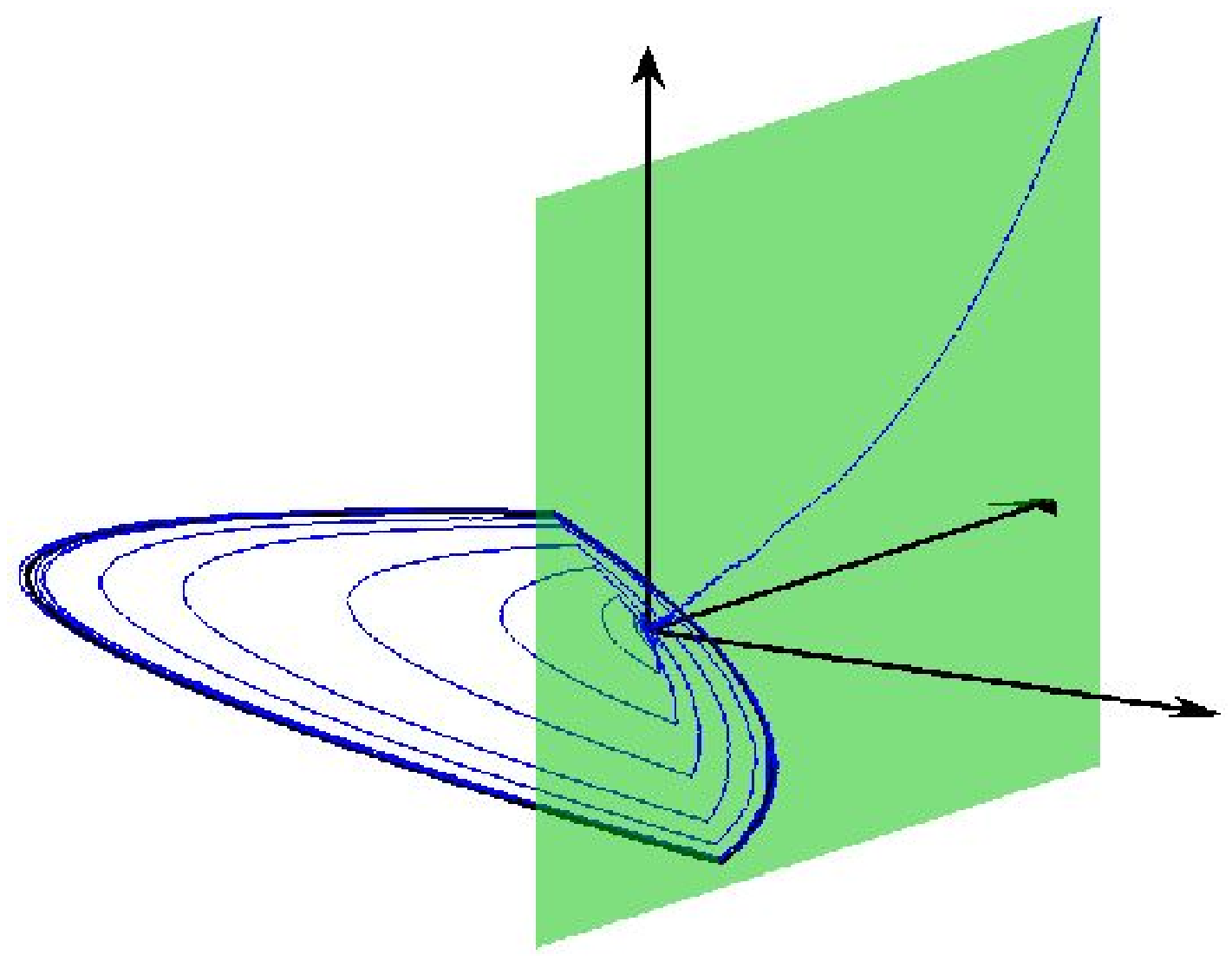}}
\put(0,0){\includegraphics[height=4cm]{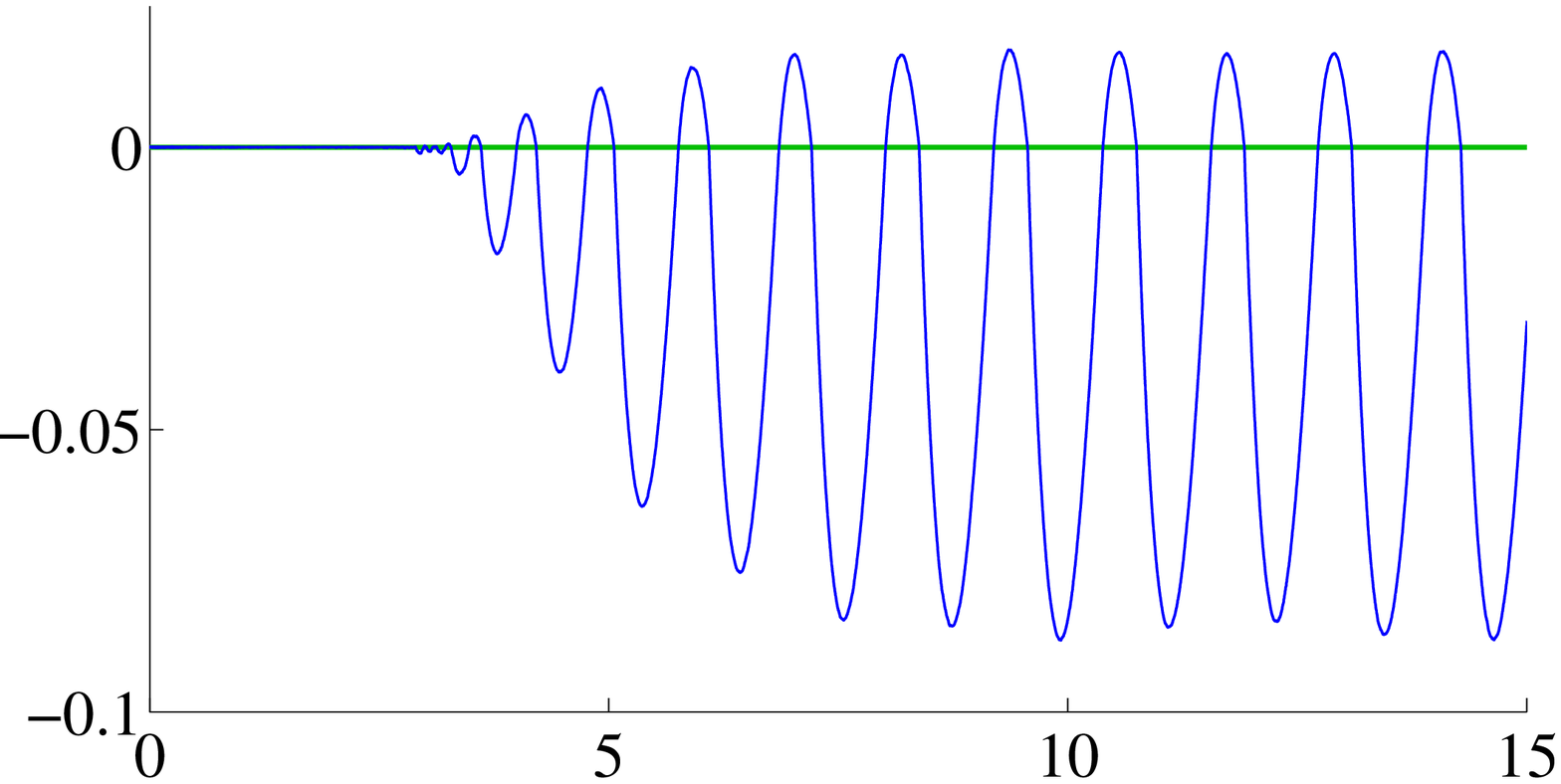}}
\put(8.5,4){\includegraphics[height=6cm]{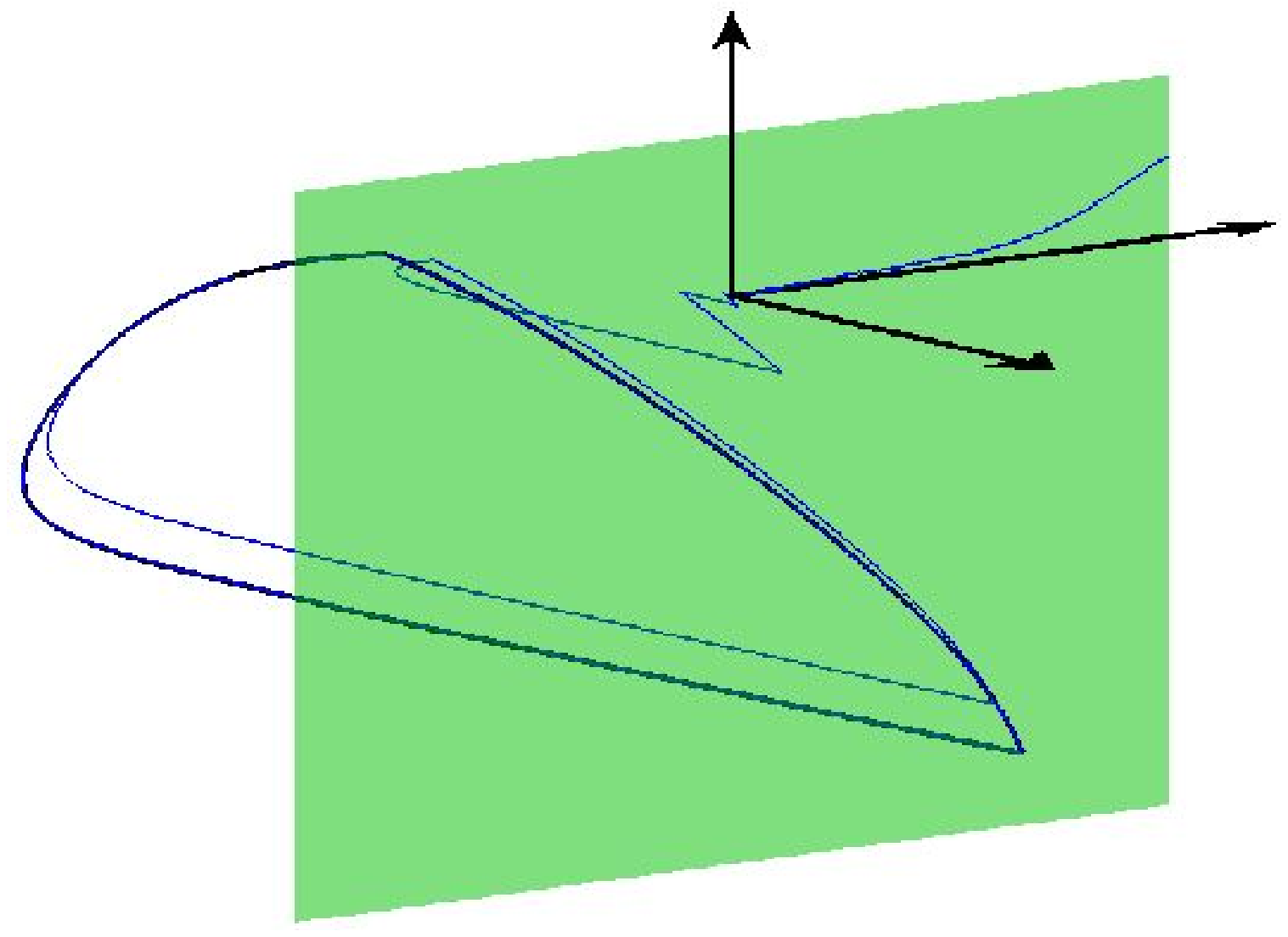}}
\put(8.5,0){\includegraphics[height=4cm]{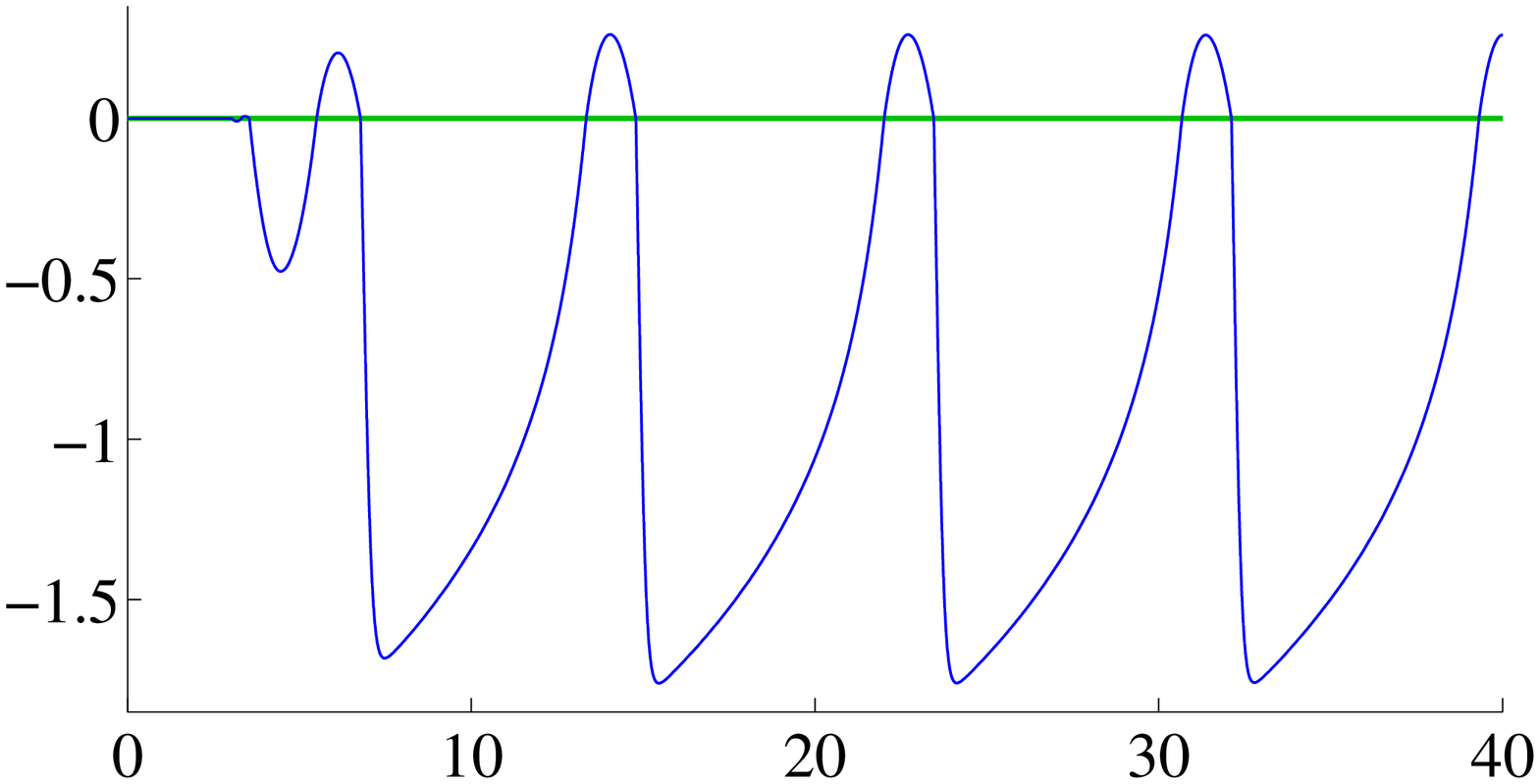}}
\put(1.1,9.5){\large \sf A}
\put(9.6,9.5){\large \sf B}
\put(7.56,5.7){\footnotesize $x$}
\put(6.5,7.05){\footnotesize $y$}
\put(4.17,9.63){\footnotesize $z$}
\put(.9,6.9){\footnotesize $\Gamma$}
\put(4.3,-.17){\small $t$}
\put(.2,3.18){\small $x$}
\put(15.3,7.64){\footnotesize $x$}
\put(16.14,8.7){\footnotesize $y$}
\put(13.29,9.8){\footnotesize $z$}
\put(9.8,8.34){\footnotesize $\Gamma$}
\put(12.8,-.17){\small $t$}
\put(8.7,3.33){\small $x$}
\end{picture}
\caption{
\label{fig:phaseqq23}
Panel A shows a typical sample solution, both in phase space and as a time series,
for (\ref{eq:sde}) with (\ref{eq:VmVpEx23}) and (\ref{eq:GLGRex2}) using $D = I$ and $\ee = 0.001$.
Panel B shows an typical sample solution using instead (\ref{eq:GLGRex3}).
Both solutions start at $X = (0,1,1)$, pass near the two-fold $X = (0,0,0)$,
then approach a stable periodic orbit $\Gamma$.
The solutions were computed using the Euler-Maruyama method
with a step-size of $\Delta t = 10^{-5}$.
}
\end{center}
\end{figure}

A sample solution to (\ref{eq:sde})
using (\ref{eq:VmVpEx23})-(\ref{eq:GLGRex2}) is shown in Fig.~\ref{fig:phaseqq23}-A. 
The initial point is $X = (0,1,1)$, and we denote this solution as $\tilde{\varphi}_{\ee}(t;0,1,1,0)$.
The deterministic solution, $\tilde{\varphi}(t;0,1,1,0)$,
slides into the two-fold in a time $t_0 = 3.0445$ (to four decimal places).
The perturbed solution $\tilde{\varphi}_{\ee}(t;0,1,1,0)$ initially follows a random path close to $\tilde{\varphi}(t;0,1,1,0)$ along the discontinuity surface, 
and at $t = t_0$ is located near the two-fold.
While $t \approx t_0$, the noise affects $\tilde{\varphi}_{\ee}(t)$ qualitatively.
After this time $\tilde{\varphi}_{\ee}(t)$ spirals outward.
The initial portion of this spiralling motion is close to the two-fold,
and so the higher order terms $G_L$ and $G_R$ have little influence.
At later times $\tilde{\varphi}_{\ee}(t)$ is located relatively far from the two-fold.
Here $G_L$ and $G_R$ are important and $\tilde{\varphi}_{\ee}(t)$ approaches
the stable periodic orbit of (\ref{eq:odePerturbed}), $\Gamma$.


We computed $10^4$ sample solutions analogous to the solution shown in Fig.~\ref{fig:phaseqq23}-A
(i.e.~with the same initial point and parameter values).
Fig.~\ref{fig:phaseHistograms}-A shows a histogram of the phase $\phi_T$ of these solutions.
Here we used $T = 15$ because transient dynamics appears to have decayed by this time.
We notice that the phase is roughly uniformly distributed on $[0,2 \pi)$.

Fig.~\ref{fig:phaseqq23}-B shows a sample solution using (\ref{eq:GLGRex3}) instead of (\ref{eq:GLGRex2}).
As in panel A, the solution approaches the two-fold
(the deterministic sliding time is $t_0 = 2.9763$, to four decimal places),
then spirals out toward a stable periodic orbit $\Gamma$.
Fig.~\ref{fig:phaseHistograms}-B shows a histogram of the phase $\phi_T$
of $10^4$ sample solutions, using $T = 40$.
Again the phase is highly random, but in this case
the phase distribution is not well-approximated by a uniform distribution.
A phase of $\phi_T \approx \frac{3 \pi}{2}$ appears to be about twice as likely
as a phase of $\phi_T \approx \frac{\pi}{2}$.
In further simulations (not shown) with different values of $\ee$ and $T$ we observed similar phase distributions
to those shown in Fig.~\ref{fig:phaseHistograms}.

In the next section we combine the polar coordinatisation of the local dynamics given in \S\ref{sub:polar},
with the global return time function $f$ for orbits spiralling out from the two-fold
given in \S\ref{sub:returnTimes},
in order to construct approximations for the phase distributions of Fig.~\ref{fig:phaseHistograms}.

\subsection{The probability distribution of the asymptotic phase}
\label{sub:pdf}

Let $X_0 \in \mathbb{R}^3$ be such that the deterministic forward orbit
of this point, $\tilde{\varphi}(t;X_0,0)$,
is located at the two-fold at some time $t_0 > 0$.
For sample solutions $\tilde{\varphi}_{\ee}(t;X_0,0)$,
we are only interested in the phase $\phi_T$ at large values of $T$
(in order to adequately approximate the asymptotic phase $\phi$)
but our analysis requires considering all $T > t_0$.

For any $T > t_0$, let $p_T(s_T)$ denote the probability density function for the value of $s_T$
(the previous time of intersection with the discontinuity surface at a point with $y > 0$).
We begin by explaining why, for intermediate values of $T-t_0$, specifically $\ee \ll T-t_0 \ll 1$,
it is suitable to assume $s_T - t_0$ has a reciprocal probability distribution, that is
\begin{equation}
p_T(s_T) = \frac{1}{\ln \left( \frac{T-t_0}{f^{-1}(T)-t_0} \right) (s_T - t_0)} \;.
\label{eq:pT}
\end{equation}

Since the noise amplitude $\ee$ is small,
if $T-t_0 \ll 1$ then an arbitrary sample solution
$\tilde{\varphi}_{\ee}(t)$ will lie near the two-fold with high probability.
Thus for the purposes of computing $s_T$
we can ignore the higher order terms $G_L$ and $G_R$.
Also if $\ee \ll T-t_0$,
then with high probability an arbitrary sample solution 
will lie sufficiently far from the two-fold that for the purposes of computing $s_T$
we can ignore the noise.
Therefore, here it is suitable to restrict our attention to the normal form (\ref{eq:ode}),
and we work with this system in polar coordinates (\ref{eq:odePolar}).

Any solution to (\ref{eq:odePolar}) that limits to the two-fold as $t \to t_0$ (with $t > t_0$)
is given by (\ref{eq:solnPolar}) for some $C \in [0,2 \pi)$.
The system (\ref{eq:odePolar}) is independent of $\theta$
which implies we should treat $C$ as a random variable with a uniform distribution on $[0,2 \pi)$.

Given $C$, we can calculate $s_T$ by using (\ref{eq:solnPolar}) to solve $\theta(t) = 0$ for $s_T$.
We write $\frac{\beta}{\alpha} \ln(T-t_0) = 2 \pi n + \hat{\theta}$,
for some $n \in \mathbb{Z}$ and $\hat{\theta} \in [0,2 \pi)$.
By (\ref{eq:solnPolar}), if $0 \le C < 2 \pi - \hat{\theta}$ then
\begin{equation}
\frac{\beta}{\alpha} \ln(s_T-t_0) + C = 2 \pi n \;,
\label{eq:sTcalc1}
\end{equation}
and if $2 \pi - \hat{\theta} \le C < 2 \pi$ then
\begin{equation}
\frac{\beta}{\alpha} \ln(s_T-t_0) + C = 2 \pi (n+1) \;.
\label{eq:sTcalc2}
\end{equation}
Since $C$ is uniformly distributed,
by (\ref{eq:sTcalc1}) and (\ref{eq:sTcalc2}), $\ln(s_T-t_0)$ is also uniformly distributed.
Therefore, in this approximation, $s_T-t_0$ has a reciprocal probability distribution
as given by (\ref{eq:pT}).

As sample solutions $\tilde{\varphi}_{\ee}(t)$ continue to evolve outward from the two-fold,
since there are no other singular points for solutions to encounter,
the noise only has the effect of diffusing intersection times $s_T$ by a small amount.
Therefore in order to approximate $p_T$ for larger values of $T > t_0$,
we can again ignore the noise (as long as $T$ is not so large
that the contribution of the accumulated small diffusive effects is significant).
In this approximation, $p_T$ can be expressed in terms of $p_{f^{-1}(T)}$ by iterating
the density under $f$, specifically,
\begin{equation}
p_T(s_T) = p_{f^{-1}(T)} \left( f^{-1}(s_T) \right) \frac{d f^{-1}}{d s_T} \;.
\label{eq:pfTs}
\end{equation}
Iterating $n$ times gives
\begin{equation}
p_T(s_T) = p_{f^{-n}(T)} \left( f^{-n}(s_T) \right) \frac{d f^{-n}}{d s_T} \;.
\label{eq:pfnTs}
\end{equation}

\begin{figure}[t!]
\begin{center}
\setlength{\unitlength}{1cm}
\begin{picture}(16.5,6)
\put(0,0){\includegraphics[height=6cm]{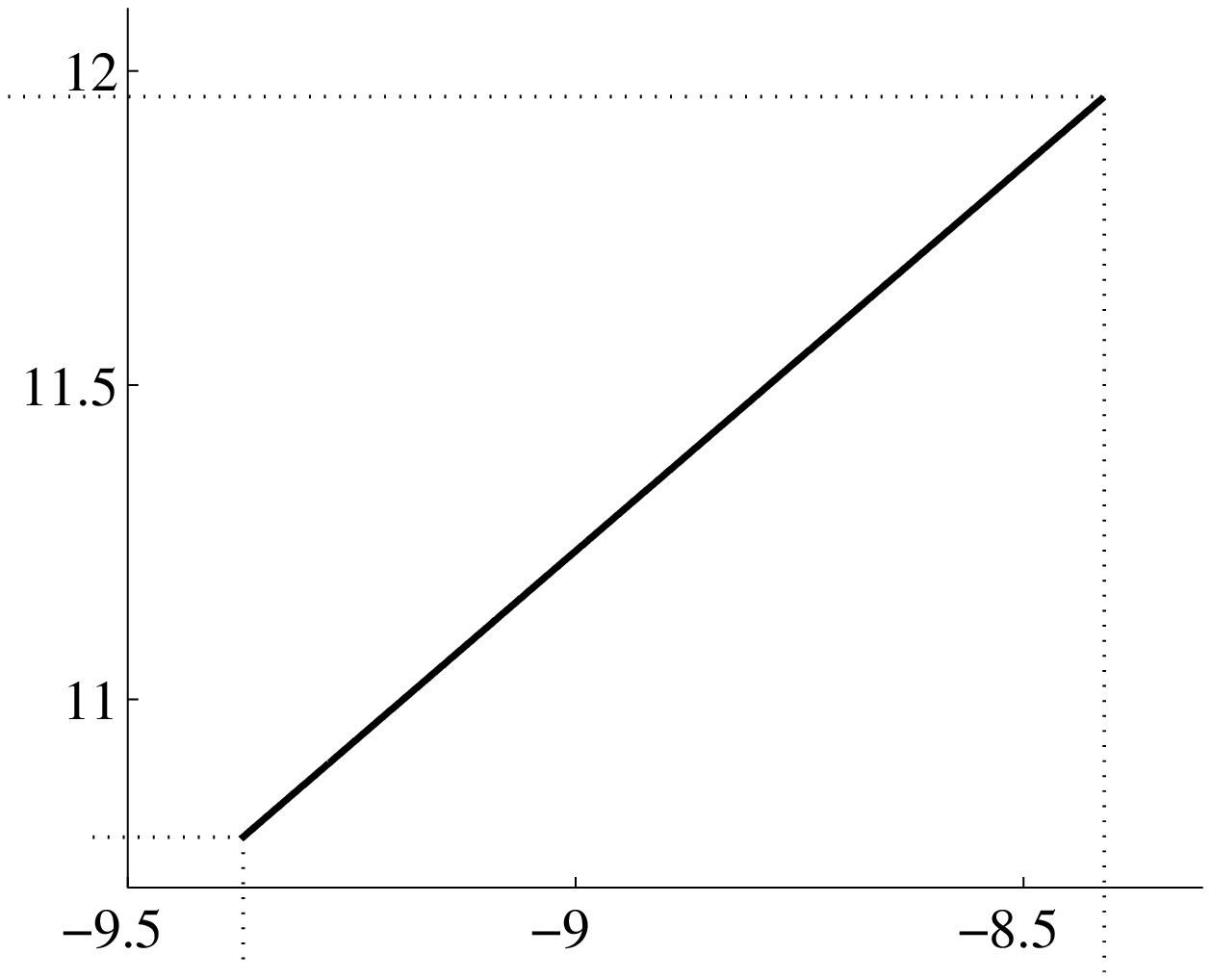}}
\put(8.5,0){\includegraphics[height=6cm]{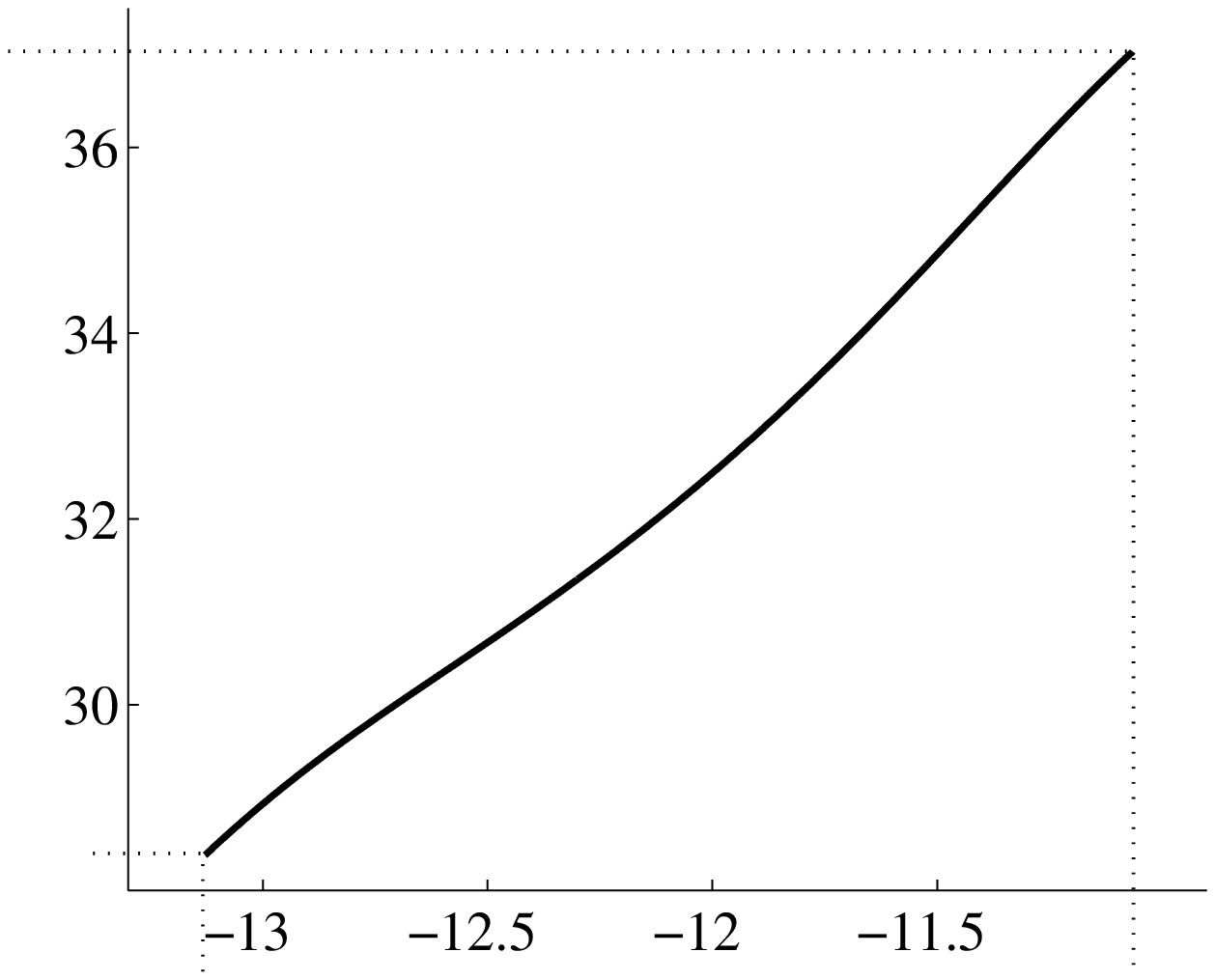}}
\put(2.2,5.8){\large \sf A}
\put(10.7,5.8){\large \sf B}
\put(1.1,.2){\scriptsize $\ln \left( f^{-n-1}(T-t_0) \right)$}
\put(5.9,.2){\scriptsize $\ln \left( f^{-n}(T-t_0) \right)$}
\put(0,1.2){\scriptsize $f^{-1}(T-t_0)$}
\put(.26,5.39){\scriptsize $T-t_0$}
\put(9.6,.2){\scriptsize $\ln \left( f^{-n-1}(T-t_0) \right)$}
\put(14.4,.2){\scriptsize $\ln \left( f^{-n}(T-t_0) \right)$}
\put(8.5,1.1){\scriptsize $f^{-1}(T-t_0)$}
\put(8.76,5.65){\scriptsize $T-t_0$}
\put(4,-.17){\small $\ln(a)$}
\put(0,3.8){\small $f^n(a)$}
\put(12.5,-.17){\small $\ln(a)$}
\put(8.5,3.8){\small $f^n(a)$}
\end{picture}
\caption{
Plots of $f^n$, using $n = 10$, where $f$ is the return time function defined in \S\ref{sub:returnTimes}.
Panel A corresponds to (\ref{eq:GLGRex2}), for which $t_0 \approx 3.0445$, and we used $T = 15$.
Panel B corresponds to (\ref{eq:GLGRex3}), for which $t_0 \approx 2.9763$, and we used $T = 40$.
\label{fig:phasefn23}
}
\end{center}
\end{figure}

We can therefore approximate $p_T$ for a large value of $T$ by using (\ref{eq:pfnTs})
with a value of $n$ that is sufficiently large that $f^{-n}(T) - t_0$ is small,
and so the reciprocal probability density function (\ref{eq:pT}) can be used for $p_{f^{-n}(T)}$.
This is the manner by which we obtained the two approximations shown in Fig.~\ref{fig:phaseHistograms}, using $n = 10$ in both cases.
The iterative procedure (\ref{eq:pfnTs}) was performed numerically
because analytic expressions for $f$ appear to be unavailable.

Fig.~\ref{fig:phasefn23} shows plots of $f^n$ for both (\ref{eq:GLGRex2}) and (\ref{eq:GLGRex3}).
We note that if $f^n(a)$ is an affine function of $\ln(a)$ then
our procedure generates a uniform distribution for $p_T$.
Indeed in Fig.~\ref{fig:phasefn23}-A $f^n$
is indistinguishable from an affine function on the given scale
and so in Fig.~\ref{fig:phaseHistograms}-A the distribution is approximately uniform.
In Fig.~\ref{fig:phasefn23}-B, $f^n$ has a noticeable nonlinearity
and for this reason the distribution in Fig.~\ref{fig:phaseHistograms}-B is significantly non-uniform.



\section{Desynchronising a collection of smooth oscillators}
\label{sec:desynch}
\setcounter{equation}{0}


Suppose we wish to break the synchrony of a collection of coordinated oscillators.
In \cite{DaHe09} this is achieved by applying a control
that pushes the state of the oscillators toward an unstable equilibrium,
where small random perturbations efficiently randomise the phase of the oscillators.
Upon removal of the control action, each oscillator returns to the original regular periodic motion
but now the oscillators are desynchronised.
Here we propose a ``fast'' method to achieve this, using a two-fold as the phase singularity rather than an equilibrium, in which case 
there is no slowing of the dynamics as the oscillators return to periodic motion.


To illustrate our method we use the Hopf bifurcation normal form
\begin{equation}
\left( \dot{x}, \dot{y} \right) = F(x,y) =
\left( x - y - x(x^2+y^2), x + y - y(x^2+y^2) \right) \;,
\label{eq:HopfNormalForm}
\end{equation}
and apply a discontinuous control that creates a two-fold singularity.
In principle this can be achieved with any system exhibiting a stable periodic orbit,
including excitable systems such as the FitzHugh-Nagumo equations \cite{DaHe09}.

Consider the two-dimensional stochastic differential equation
\begin{equation}
\big( dx(t), dy(t) \big) =
\big( F(x(t),y(t)) + H(t-t_1) H(t_2-t) c(t) \big) \,dt + \ee \,dW(t) \;,
\label{eq:stochHopfNormalForm}
\end{equation}
where $W(t)$ is a standard two-dimensional Brownian motion,
and $0 \le \ee < 1$ is the noise amplitude.
$H$ is the Heaviside function
and $t_1$ and $t_2$ are the start and end times of a control action given by
\begin{equation}
c(t) = \begin{cases}
(a_1 t, a_2) \;, & x \le 0 \\
(a_3 t, a_4) \;, & x > 0
\end{cases} \;,
\label{eq:c}
\end{equation}
where $a_1, a_2, a_3, a_4 \in \mathbb{R}$ are control parameters.


By treating the time $t$ as a third dynamic variable, i.e.~with $\dot{t} = 1$,
the system is three-dimensional and
is piecewise-smooth when $t_1 < t < t_2$ with the discontinuity surface $x = 0$.
The fold lines on $x = 0$ (the lines where $\dot{x} = 0$) 
are $y = a_1 t$ 
and $y = a_3 t$, 
so the origin $(x,y,t) = (0,0,0)$ is a two-fold.
Via elementary calculations we find we need
$a_2 < a_1$, $a_3 < a_4$ and $a_1 \ne a_3$ in order for the two-fold
to be generic and for both folds to be invisible, as in (\ref{eq:ode}).
We also require $a_1 < a_3$ in order to have $V^-,V^+ < 0$ and $V^- V^+ > 1$ (\ref{eq:VmVp}).
In summary, we require
\begin{equation}
a_2 < a_1 < a_3 < a_4 \;.
\label{eq:controlParametersConstraint}
\end{equation}


\begin{figure}[t!]
\begin{center}
\setlength{\unitlength}{1cm}
\begin{picture}(15.5,4.5)
\put(0,0){\includegraphics[height=4.5cm]{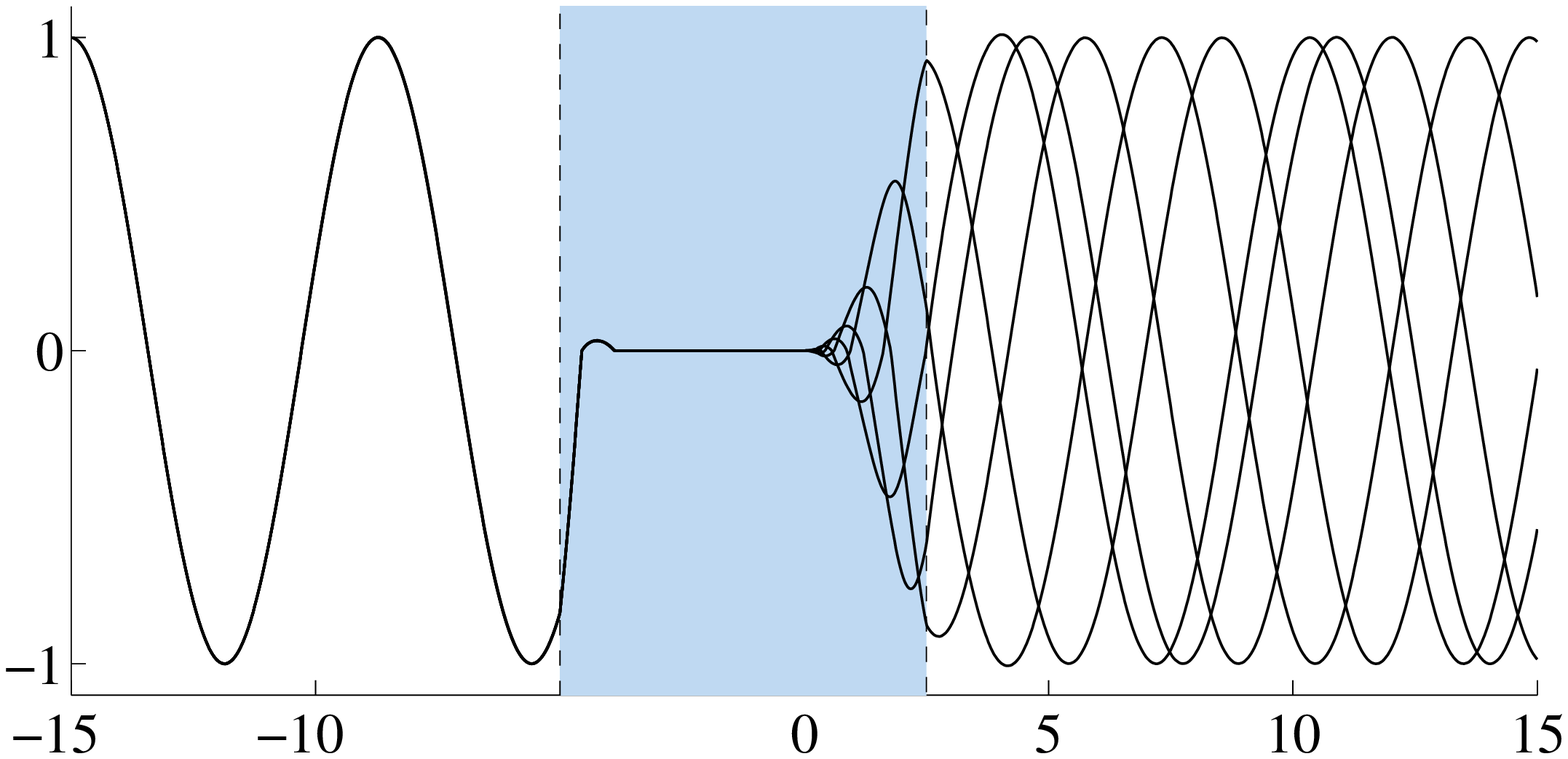}}
\put(9.5,0){\includegraphics[height=4.5cm]{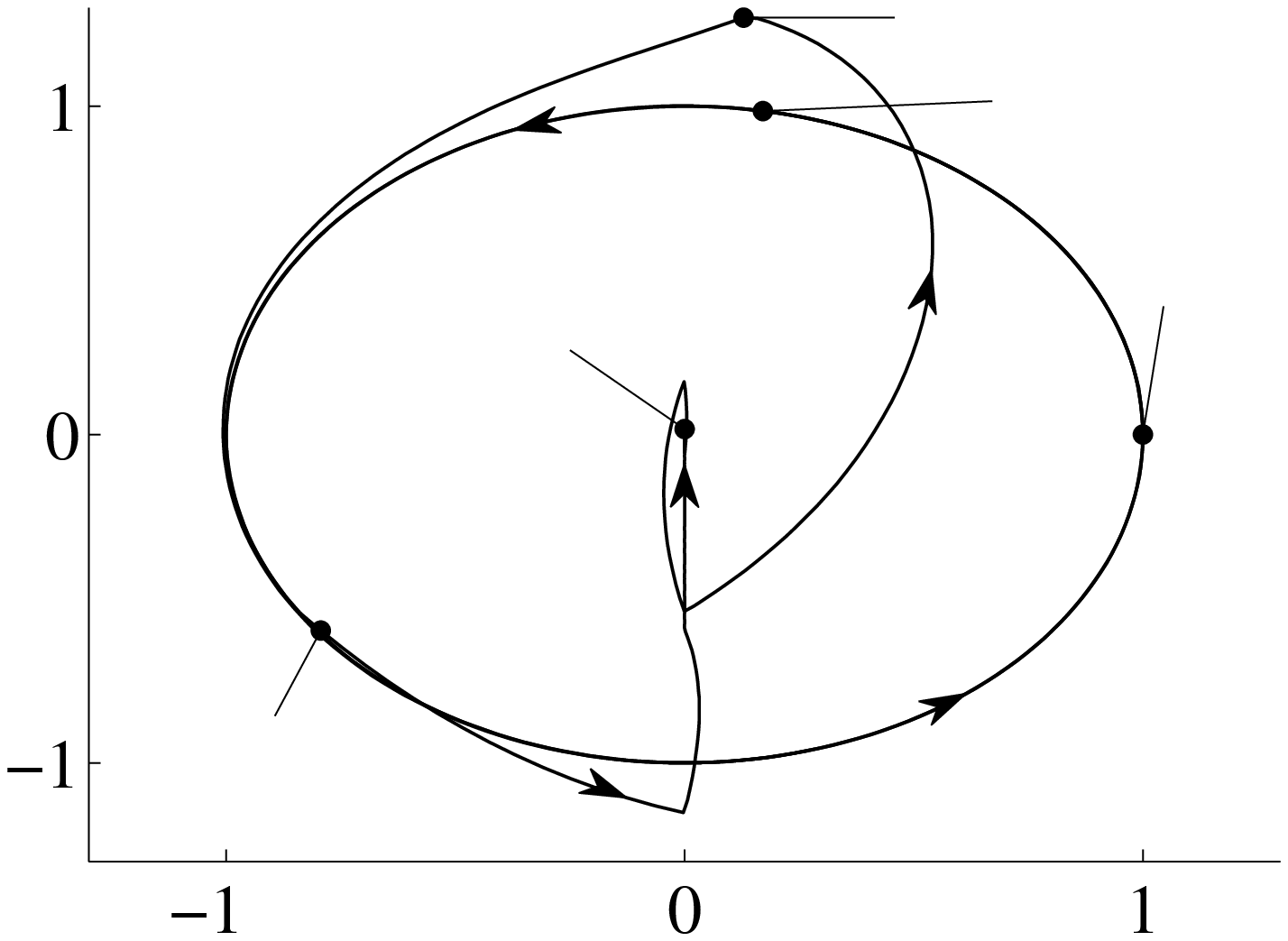}}
\put(1,4.4){\large \sf A}
\put(10.5,4.4){\large \sf B}
\put(4.59,-.04){\small $t$}
\put(0,2.43){\small $x$}
\put(3.2,.28){\footnotesize $t_1$}
\put(5.27,.28){\footnotesize $t_2$}
\put(3.55,3.9){\footnotesize control on}
\put(12.63,0){\small $x$}
\put(9.5,2.45){\small $y$}
\put(14.75,3.2){\tiny $t=-15$}
\put(10.6,1.05){\tiny $t=t_1$}
\put(11.83,2.97){\tiny $t=0$}
\put(13.8,4.37){\tiny $t=t_2$}
\put(14.2,4){\tiny $t=15$}
\end{picture}
\caption{
Panel A shows five sample solutions to (\ref{eq:stochHopfNormalForm})
using (\ref{eq:a1a2a3a4}), (\ref{eq:t1t2}) and $\ee = 0.001$,
and the initial condition $(x,y) = (1,0)$ at $t = -15$.
The control (\ref{eq:c}) causes the solutions to pass close to a two-fold at $t = 0$,
effectively randomising the phase.
Panel B shows one of the solutions in phase space,
with the location of the solution highlighted at key times.
The sample solutions were computed using the Euler-Maruyama method
with a step-size of $\Delta t = 10^{-5}$.
\label{fig:sampleHopf}
} 
\end{center}
\end{figure}

Fig.~\ref{fig:sampleHopf} illustrates the phase randomisation using
\begin{equation}
a_1 = -0.2 \;, \qquad
a_2 = -1 \;, \qquad
a_3 = 0.2 \;, \qquad
a_4 = 1 \;,
\label{eq:a1a2a3a4}
\end{equation}
and
\begin{equation}
t_1 = -5 \;, \qquad
t_2 = 2.5 \;.
\label{eq:t1t2}
\end{equation}
In simulations we found that with relatively small values of the $a_i$,
a large negative value of $t_1$ is required in order
for the control to direct orbits into the two-fold, 
and that with large values of $a_i$
the distributions of the asymptotic phase were substantially non-uniform.


\section{Discussion}
\label{sec:conc}
\setcounter{equation}{0}

In this paper we have studied orbits that pass through an invisible two-fold
(Teixeira singularity) then limit to a stable periodic orbit.
We considered the presence of small noise
in order to remove the ambiguity of evolution through the two-fold.
We found that the phase of the limiting periodic motion is highly randomised
and that the probability distribution of the phase depends crucially on the nature
of the nonlinear dynamics experienced by orbits between the two-fold and the periodic orbit.
By using polar coordinates to describe the motion of orbits as they initially depart from the two-fold,
and a one-dimensional return map to describe the global dynamics approaching the periodic orbit,
we constructed approximations to the probability density function of the asymptotic phase
that matched well to the results of numerical simulations, Fig.~\ref{fig:phaseHistograms}.

In \S\ref{sec:desynch} we showed how a simple discontinuous control law
can generate a two-fold in a smooth system.
We propose that in this fashion the two-fold can be used to
desynchronise a collection of oscillators and that this may have the following advantages to
desynchronisation using an unstable equilibrium.

\begin{enumerate}
\renewcommand{\labelenumi}{\roman{enumi})}
\item
With an unstable equilibrium,
the control action must direct the state of each oscillator to near the equilibrium,
and the effectiveness of the randomisation correlates with the accuracy of this action.
With a two-fold, however, the control does not need not be as precise because it
only needs to direct orbits to the basin of attraction of the two-fold.
\item
With an unstable equilibrium,
the effect of the randomisation is due to inherent stochasticity in the system
and any artificial randomness in the control law.
With a two-fold, the randomisation is inherent in the 
ambiguity of forward evolution through the two-fold.
\item
Orbits dwell near unstable equilibria and
so can take a relatively long time to return to regular periodic motion.
With a two-fold there is no slowing of the dynamics,
hence we refer to our desynchronisation as ``fast'' phase randomisation.
Indeed, in Fig.~\ref{fig:sampleHopf}
we see that solutions return quickly to approximately
regular periodic motion after the control is turned off.
Whether this truly constitutes fast desynchronisation requires a study of the full time the control action is required to act for. This, and a study of practical control actions, 
remain for future work.
\end{enumerate}

Teixeira's paper \cite{Te90} inspired a legacy of intrigue around a specific case of the two-fold singularity, first studied more generally in \cite{Fi88}. Interest has only increased as its role as a determinacy-breaking singularity has become more clear \cite{Je11,CoJe11}. This paper begins to reveal the practical side of these insights, including how two-folds may manifest in physical systems, and how they might be put to use as tools for control. Little is still known about where two-folds appear naturally in physical systems. Similarly little is known about what applications they may have in control systems, but a role as a phase randomiser is suggested here.
Perhaps the obvious next step is to design implementable control circuits to investigate the practical obstacles and opportunities that they present.

\appendix

\section{Proof of Proposition \ref{pr:polar}}
\label{app:polarDerivation}
\setcounter{equation}{0}

\begin{proof}
The $y$-value of $\psi_a(a)$ is given by (\ref{eq:yValueOnZeta}),
and so by (\ref{eq:solnPolarAlt}) the $r$-value of $\psi_a(a)$ is equal to $\alpha a$.
Since $\psi_a(a) \in \zeta$, the $y$ and $r$-values of $\psi_a(a)$ are the same,
see (\ref{eq:theta0}), hence $\alpha$ is given by (\ref{eq:alphabeta}).

At the start of \S\ref{sub:polar} we showed
that, given $x < 0$ and $y \in \mathbb{R}$,
if $\tau_L > 0$ is given by (\ref{eq:tauL}) and $y_0 = y - V^- \tau_L$,
then evolving (\ref{eq:ode}) forward from $(0,y_0,\gamma y_0)$ for a time $\tau_L$
takes us to the point $(x,y,z)$,
where $x = \frac{-1}{V^-} \Xi(y,z)$ (\ref{eq:Xi}).
Moreover, this defines a bijection from the region $x < 0$, $y \in \mathbb{R}$
to the region $y_0 > 0$, $0 < \tau_L < -2 \gamma y_0$.

By (\ref{eq:odePolar}), evolving $(r,\theta) = (y_0,0)$ for a time $\tau_L$ takes us to the point
\begin{equation}
(r,\theta) = \left( y_0 + \alpha \tau_L ,\, \frac{\beta}{\alpha}
\ln \left( \frac{\alpha \tau_L}{y_0} + 1 \right) \right) \;.
\label{eq:rThetaLeft}
\end{equation}
This is a bijection from $y_0 > 0$, $0 < \tau_L < -2 \gamma y_0$
to $r > 0$, $0 < \theta < \frac{\beta}{\alpha} \ln(1 - 2 \alpha \gamma)$
(assuming $\beta > 0$).
By substituting (\ref{eq:tauL}) and $y_0 = y - V^- \tau_L$ into (\ref{eq:rThetaLeft}) and simplifying,
we arrive at (\ref{eq:r}) and (\ref{eq:theta}) for $x < 0$.

Similarly, given $x > 0$ and $y \in \mathbb{R}$,
let $\tau_R < 0$ be given by (\ref{eq:tauR}) and $y_0 = y - \tau_R$.
This is a bijection to $y_0 > 0$, $-2 y_0 < \tau_R < 0$,
and evolving (\ref{eq:ode}) backward from $(0,y_0,\gamma y_0)$ for a time $\tau_R$
takes us to $(x,y,z)$,
where $x = \frac{1}{V^+} \Xi(y,z)$ (\ref{eq:Xi}).

In polar coordinates evolving $(r,\theta) = (y_0,0)$ for a time $\tau_R$ takes us to
\begin{equation}
(r,\theta) = \left( y_0 + \alpha \tau_R ,\, \frac{\beta}{\alpha}
\ln \left( \frac{\alpha \tau_R}{y_0} + 1 \right) + 2 \pi \right) \;,
\label{eq:rThetaRight}
\end{equation}
which is a bijection to $r > 0$, $\frac{\beta}{\alpha} \ln(1 - 2 \alpha) + 2 \pi < \theta < 2 \pi$.
Substituting (\ref{eq:tauR}) and $y_0 = y - \tau_R$
into (\ref{eq:rThetaRight}) leads to (\ref{eq:r}) and (\ref{eq:theta}) for $x > 0$.

Finally, by matching the limiting left and right values for $\theta$ on the negative $y$-axis
as given by (\ref{eq:theta}), we obtain
\begin{equation}
\theta = \frac{\beta}{\alpha} \ln(1 - 2 \alpha \gamma) =
\frac{\beta}{\alpha} \ln(1 - 2 \alpha) + 2 \pi \;.
\label{eq:negativeyaxis}
\end{equation}
By using the identity $\mu = \frac{1 - 2 \alpha \gamma}{1 - 2 \alpha}$,
by solving (\ref{eq:negativeyaxis}) for $\beta$
we obtain the expression for $\beta$ given in (\ref{eq:alphabeta}).

This completes the proof because the last calculation provides continuity of $\theta(x,y)$,
and by construction $r$ and $\theta$ satisfy the ordinary differential equations (\ref{eq:odePolar}).
\end{proof}

\section*{Acknowledgements}

MRJ's research is supported by EPSRC Fellowship grant EP/J001317/2.


\begin{thebibliography}{10}

\bibitem{HaBe07}
C.~Hammond, H.~Bergman, and P.~Brown.
\newblock Pathological synchronization in {P}arkinson's disease: networks,
  models and treatments.
\newblock {\em Trends Neurosci.}, 30(7):357--364, 2007.

\bibitem{NiFe95}
A.~Nini, A.~Feingold, H.~Slovin, and H.~Bergman.
\newblock Neurons in the globus pallidus do not show correlated activity in the
  normal monkey, but phase-locked oscillations appear in the {MPTP} model of
  {P}arkinsonism.
\newblock {\em J. Neurophysiol.}, 74(4):1800--1805, 1995.

\bibitem{Ta02}
P.A. Tass.
\newblock Effective desynchronisation with bipolar double-pulse stimulation.
\newblock {\em Phys. Rev. E}, 66:036226, 2002.

\bibitem{Ta03}
P.A. Tass.
\newblock A model of desynchronizing deep brain stimulation with a
  demand-controlled coordinated reset of neural subpopulations.
\newblock {\em Biol. Cybern.}, 89:81--88, 2003.

\bibitem{PoHa06}
O.V. Popovych, C.~Hauptmann, and P.A. Tass.
\newblock Desynchronization and decoupling of interacting oscillators by
  nonlinear delayed feedback.
\newblock {\em Int. J. Bifurcation Chaos}, 16(7):1977--1987, 2006.

\bibitem{FrCh12}
A.~Franci, A.~Chaillet, E.~Panteley, and F.~Lamnabhi-Lagarrigue.
\newblock Desynchronization and inhibition of {K}uramoto oscillators by scalar
  mean-field feedback.
\newblock {\em Math. Control Signals Syst.}, 24:169--217, 2012.

\bibitem{DaHe09}
P.~Danzl, J.~Hespanha, and J.~Moehlis.
\newblock Event-based minimum-time control of oscillatory neuron models.
  {P}hase randomization, maximal spike rate increase, and desynchonization.
\newblock {\em Biol. Cybern.}, 101:387--399, 2009.

\bibitem{Fi88}
A.F. Filippov.
\newblock {\em Differential Equations with Discontinuous Righthand Sides.}
\newblock Kluwer Academic Publishers., Norwell, 1988.

\bibitem{Te90}
M.A. Teixeira.
\newblock Stability conditions for discontinuous vector fields.
\newblock {\em J. Differential Equations}, 88:15--29, 1990.

\bibitem{Je11}
M.R. Jeffrey.
\newblock Nondeterminism in the limit of nonsmooth dynamics.
\newblock {\em Phys. Rev. E}, 106:254103, 2011.

\bibitem{CoJe13}
A.~Colombo and M.R. Jeffrey.
\newblock The two-fold singularity of non-smooth flows: {L}eading order
  dynamics in $n$-dimensions.
\newblock {\em Phys. D}, 263:1--10, 2013.

\bibitem{DiCo11}
M.~di~Bernardo, A.~Colombo, and E.~Fossas.
\newblock Two-fold singularity in nonsmooth electrical systems.
\newblock In {\em IEEE International Symposium on Circuits and Systems.}, pages
  2713--2716, 2011.

\bibitem{SzJe14}
R.~Szalai and M.R. Jeffrey.
\newblock Non-deterministic dynamics of a mechanical system.
\newblock {\em Phys. Rev. E}, 90:022914, 2014.

\bibitem{DeJe11}
M.~Desroches and M.R. Jeffrey.
\newblock Canards and curvature: nonsmooth approximation by pinching.
\newblock {\em Nonlinearity}, 24:1655--1682, 2011.

\bibitem{JeCo09}
M.R. Jeffrey and A.~Colombo.
\newblock The two-fold singularity of discontinuous vector fields.
\newblock {\em SIAM J. Appl. Dyn. Sys.}, 8(2):624--640, 2009.

\bibitem{CoJe11}
A.~Colombo and M.R. Jeffrey.
\newblock Nondeterministic chaos, and the two-fold singularity in piecewise
  smooth flows.
\newblock {\em SIAM J. Appl. Dyn. Sys.}, 10(2):423--451, 2011.

\bibitem{FeGa12}
S.~Fern\'{a}ndez-Garc\'{i}a, D.A. Garc\'{i}a, G.O. Tost, M.~di~Bernardo, and
  M.R. Jeffrey.
\newblock Structural stability of the two-fold singularity.
\newblock {\em SIAM J. Appl. Dyn. Syst.}, 11(4):1215--1230, 2012.

\bibitem{Si14c}
D.J.W. Simpson.
\newblock On resolving singularities of piecewise-smooth discontinuous vector
  fields via small perturbations.
\newblock {\em Discrete Contin. Dyn. Syst.}, 34(9):3803--3830, 2014.

\bibitem{Gu75}
J.~Guckenheimer.
\newblock Isochrons and phaseless sets.
\newblock {\em J. Math. Biol.}, 1:259--273, 1975.

\bibitem{Wi01}
A.T. Winfree.
\newblock {\em The Geometry of Biological Time.}
\newblock Springer, New York, 2001.

\bibitem{OsMo10}
H.M. Osinga and J.~Moehlis.
\newblock Continuation-based computation of global isochrons.
\newblock {\em SIAM J. Appl. Dyn. Syst.}, 9(4):1201--1228, 2010.

\bibitem{LaKr14}
P.~Langfield, B.~Krauskopf, and H.M. Osinga.
\newblock Solving {W}infree's puzzle: {T}he isochrons in the {FitzHugh-Nagumo}
  model.
\newblock {\em Chaos}, 24:013131, 2014.

\end{thebibliography}

\end{document}